\documentclass{elsarticle}  

\usepackage{amsmath}
\usepackage{amsfonts}
\usepackage{amssymb}
\usepackage{graphicx}
\usepackage[utf8]{inputenc}
\usepackage[dvipsnames]{xcolor}
\usepackage[shortlabels]{enumitem}
\usepackage[utf8]{inputenc}
\usepackage{amssymb,amsmath,amsthm, amsfonts}
\usepackage{mathtools}
\usepackage[algo2e,linesnumbered,noend,noline]{algorithm2e}
\usepackage{algorithmic,refcount}
\usepackage[a4paper]{geometry}

\newtheorem{corollary}{Corollary}
\newtheorem{theorem}{Theorem}
\newtheorem{lemma}{Lemma}
\newtheorem{claim}{Claim}

\usepackage{lipsum}
\usepackage{amsfonts}
\usepackage{graphicx}
\usepackage{epstopdf}
\usepackage{algorithmic}
\usepackage{titlesec}
\usepackage{hyperref}

\ifpdf
  \DeclareGraphicsExtensions{.eps,.pdf,.png,.jpg}
\else
  \DeclareGraphicsExtensions{.eps}
\fi

\title{Reachability in arborescence packings}

\author{Florian H\"orsch}
 \ead{Florian.Hoersch@grenoble-inp.fr}
\address{Technische Universit\"at Ilmenau, Fakult\"at f\"ur Mathematik und Naturwissenschaften Institut f\"ur Mathematik.}
\author{Zolt\'an Szigeti}
\ead{ Zoltan.Szigeti@grenoble-inp.fr}
\address{Univ.~Grenoble~Alpes, Grenoble INP, CNRS, G-SCOP, 46 Avenue F\'elix Viallet, Grenoble, France, 38000.  }

\usepackage{amsopn}

\ifpdf
\hypersetup{
  pdftitle={Reachability in arborescence packings},
  pdfauthor={F.Hoersch and Z.Szigeti}
}
\fi

\titleclass{\subsubsubsection}{straight}[\subsection]
\newcounter{subsubsubsection}[subsubsection]
\renewcommand\thesubsubsubsection{\thesubsubsection.\arabic{subsubsubsection}}

\titleformat{\subsubsubsection}
  {\normalfont\normalsize\bfseries}{\thesubsubsubsection}{1em}{}
\titlespacing*{\subsubsubsection}
{0pt}{3.25ex plus 1ex minus .2ex}{1.5ex plus .2ex}

\makeatletter
\renewcommand\paragraph{\@startsection{paragraph}{5}{\z@}%
  {3.25ex \@plus1ex \@minus.2ex}%
  {-1em}%
  {\normalfont\normalsize\bfseries}}
\renewcommand\subparagraph{\@startsection{subparagraph}{6}{\parindent}%
  {3.25ex \@plus1ex \@minus .2ex}%
  {-1em}%
  {\normalfont\normalsize\bfseries}}
\def\toclevel@subsubsubsection{4}
\def\toclevel@paragraph{5}
\def\toclevel@paragraph{6}
\def\l@subsubsubsection{\@dottedtocline{4}{7em}{4em}}
\def\l@paragraph{\@dottedtocline{5}{10em}{5em}}
\def\l@subparagraph{\@dottedtocline{6}{14em}{6em}}
\makeatother

\begin{document}

\begin{abstract}
 Fortier et al. proposed several research problems on packing arborescences and settled some of them. Others were later solved  by Matsuoka and Tanigawa and by Gao and Yang. The last open problem is settled in this article. We show how to turn an inductive idea used in the latter two articles into a simple proof technique that allows to relate previous results on arborescence packings. 
We prove that a strong version of Edmonds' theorem on packing spanning arborescences implies Kamiyama, Katoh and Takizawa's result on packing reachability arborescences and that Durand de Gevigney, Nguyen and Szigeti's theorem on matroid-based packing of arborescences implies Kir\'aly's result on matroid-reachability-based packing of arborescences. 
Further, we deduce a new result on matroid-reachability-based packing of mixed hyperarborescences from a theorem on matroid-based packing of mixed hyperarborescences due to Fortier et al.. 
Finally, we deal with the algorithmic aspects of the problems considered. We first obtain algorithms to find the desired packings of arborescences in all settings and then apply Edmonds' weighted matroid intersection algorithm to also find solutions minimizing a given weight function.
\end{abstract}

\maketitle



\section{Introduction}

This article deals with the packing of arborescences. We focus on concluding characterizations of graphs admitting a packing of reachability arborescences from the corresponding theorems for spanning arborescences in several settings. We first give an overview of the results in this article. All technical terms which are not defined here will be explained in Section \ref{defi}.

In 1973, Edmonds \cite{e} characterized digraphs having a packing of $k$ spanning $r$-arborescences for some $k\in \mathbb Z_+$ and for some vertex $r.$ Since then, there have been numerous generalizations of this result. A first attempt is to allow different roots for the arborescences. A version with arbitrary, fixed roots can easily be derived from the theorem of Edmonds. This generalization has a significant deficiency occurring when some vertex  is not reachable from some designated root. In this case, the only information it provides is that the desired packing does not exist. 
A concept to overcome this problem has been developed by Kamiyama, Katoh and Takizawa in \cite{kkt}. Given a digraph $D$ and a multiset $R$ of the vertex set of $D$, they wish to find  a packing $\{B_r\}_{r\in R}$ of arborescences such that each $B_r$ has root $r$ and contains all the vertices reachable in $D$ from $r$. They provide a characterization of these graph-multiset pairs. Their result has been considered a breakthrough. As one of the main contributions of this article, we reprove their theorem by a very simple reduction from a stronger form of Edmonds' theorem.

Another way of generalizing the requirements on the packing of arborescences was introduced by Durand de Gevigney, Nguyen and Szigeti in \cite{dns}.  Given a digraph $D$ and a matroid ${\cal M}$ on a set of roots  $R$ in the vertex set of $D$, they now wish to find a packing  of arborescences containing an $r$-arborescence for every $r \in R$ such that every vertex of $D$ that is not in $R$ belongs to a subset of the arborescences in the packing whose roots form a basis of ${\cal M}$. A characterization of graph-matroid pairs admitting such a packing of arborescences was found in \cite{dns}.
A natural combination of the two aforementioned generalizations was introduced by Kir\'aly \cite{k}.  
 Given a digraph $D$ and a matroid ${\cal M}$ on a set $R$ of roots in the vertex set of $D$, he wishes to find  a packing  of arborescences containing an $r$-arborescence for every $r \in R$ such that every vertex $v$ of $D$ that is not in $R$ belongs  to a subset of the arborescences in the packing whose roots form a basis of the matroid ${\cal M}|R_v$ where $R_v$ is the set of roots in $R$ from which $v$ is reachable in $D$ and ${\cal M}|R_v$ is the restriction of $\cal M$ to $R_v$.
A characterization of the graph-matroid pairs admitting such a packing of arborescences was provided by Kir\'aly \cite{k}. We reprove this theorem by concluding it  from  the theorem in \cite{dns}. We wish to mention that the last two results are presented in a form which is different from the original one in \cite{dns} and \cite{k}, but both versions are equivalent.

Finally, there are attempts to also generalize the objects considered from digraphs to more general objects like mixed graphs or dypergraphs. We consider a concept unifying all of these generalizations where we want to find a matroid-reachability based packing of mixed hyperarborescences in a matroid-rooted mixed hypergraph. We derive a characterization of these mixed hypergraph-matroid pairs from a characterization for the existence of a matroid-based packing of  mixed hyperarborescences in a matroid-rooted mixed hypergraph by Fortier et al. in \cite{fklst}. This generalizes results of Gao and Yang \cite{gy} and Matsuoka and Tanigawa \cite{mt}.

A simple inductive  proof technique is presented in three more and more general frameworks.
With an increasing generality of the results, the proofs become more and more involved.
In order to be able to understand the proof of the new result of this article we recommend the reader to read all the proofs in the given order.

In the last part, we deal with the algorithmic aspects of our work. We first show that in  the first two settings, our proofs can be turned into polynomial time recursive algorithms for finding the desired packings. In order to obtain algorithms that compute optimal weighted solutions, we rely on matroid intersection. This is useful due to a result of Edmonds that states that a common independent set of two matroids of a given size minimizing a given weight function can be computed in polynomial time if polynomial time independence oracles for the matroids are available \cite{e2}. A connection in the basic setting between arborescence packing and matroid intersection has first been observed by Edmonds \cite{e4}. Later, Kir\'aly, Szigeti and Tanigawa managed to also model matroid-based packings of arborescences as the intersection of two matroids yielding an algorithm for the weighted version of this problem \cite{kst}. In a first step, we show how to extend their construction to mixed hypergraphs. In \cite{kst}, also a rather involved construction was given that modeled matroid-reachability-based  packings of arborescences as the intersection of two matroids. A similar construction for matroid-reachability-based packings of mixed hyperarborescences in matroid-rooted mixed hypergraphs has recently been given by Kir\'aly \cite{kiralgo}. We give a different algorithm for the weighted version of this problem. Making use of the proof of our result characterizing matroid-rooted mixed hypergraphs admitting a matroid-reachability-based packing of mixed hyperarborescences, we give a polynomial time algorithm that uses the algorithm that computes a solution of minimum weight  for the case of matroid-based packings of mixed hyperarborescences and only uses elementary arguments otherwise. This allows to avoid the very involved matroid construction the reachability condition imposes.

Section \ref{defi} contains the definitions we need in this article.
Section \ref{res} provides a more technical and detailed overview of the results considered. 
Section \ref{red} gives the reductions that yield our new proofs. 
Section \ref{alg} deals with the algorithmic impacts of our results.

\section{Definitions}\label{defi}

In this section, we provide the definitions and notation needed in the article. 
All graphs and matroids and graphs considered in this article are finite.
For a single element set $\{s\}$, we use {\boldmath$s$}.

Let {\boldmath $S$} be a finite set. We use {\boldmath$2^S$} to denote the power set of $S$. Let $f:2^S\rightarrow \mathbb{Z}$ be a function. Then $f$ is called {\it submodular} if $f(X)+f(Y)\geq f(X \cup Y)+f(X \cap Y)$ for all $X,Y \subseteq S$ and {\it modular} if $f(X)+f(Y)= f(X \cup Y)+f(X \cap Y)$ for all $X,Y \subseteq S$. Further, $f$ is called {\it non-decreasing} if $f(X)\leq f(Y)$ for all $X,Y\subseteq S$ with $X \subseteq Y$.


For basic notions of matroid theory, we refer to \cite{book}, Chapter 5. We only give the following less common notation. Given a matroid $\mathcal{M}$ on a set $R$ with rank function $r_{\mathcal{M}}$ and a subset  $X$ of $R$, we denote by {\boldmath$\mathcal{M}|X$} the restriction of $\mathcal{M}$ to $X$. By a {\it basis} of $X$ we mean a basis of $\mathcal{M}|X$.

\subsection{Directed graphs}

We first provide some basic notation on {\it directed graphs (digraphs)}. Let $D=(V,A)$ be a digraph. For disjoint $X,Y \subseteq V$, we denote the set of arcs with tail in $X$ and head in $Y$ by {\boldmath $\rho_A(X,Y)$} and $|\rho_A(X,Y)|$ by {\boldmath $d_A(X,Y)$}. We use {\boldmath$\rho_A^+(X)$} for $\rho_A(X,V-X)$, {\boldmath$\rho_A^-(X)$} for $\rho_A(V-X,X)$, {\boldmath$d_A^+(X)$} for $|\rho_A^+(X)|$ and {\boldmath$d_A^-(X)$} for $|\rho_A^-(X)|$. We denote by {\boldmath$N_D^+(X)$} ({\boldmath$N_D^-(X)$}) the set of vertices $v \in V-X$ such that $D$ contains at least one arc from $X$ to $v$ (from $v$ to $X$).  We call a vertex $r$ of $D$ a {\it root} in $D$ if $d_A^-(r)=0$ and a {\it simple} root if additionally $d_A^+(r)\le 1$. This name is motivated by the role these vertices will play as the roots of arborescences.

For $u,v \in V$, we say that $v$ is {\it reachable} from $u$ in $D$ if there exists a directed path from $u$ to $v$ in $D.$ For $X \subseteq V$, we denote by {\boldmath $U_X^D$} the set of vertices which are reachable in $D$ from at least one vertex in $X$, by {\boldmath$P^D_X$}  the set of vertices from which at least one vertex in $X$ is reachable  in $D$ and by {\boldmath $D[X]$} the subdigraph of $D$ induced on $X$.

An {\it arborescence} is a digraph such that exactly one arc enters each vertex except one which is a root and each vertex is reachable from this root. Observe that every arborescence contains a unique root. An arborescence whose unique root is a vertex $r$ is also called an {\it $r$-arborescence}. An arborescence $B$ is said to {\it span} $V(B)$.  By a {\it packing of arborescences} or {\it arborescence packing} in $D$, we mean a set of arc-disjoint arborescences in $D$.

We define a {\it (simply) rooted digraph} as a digraph $D=(V\cup R,A)$ with {\boldmath $R$} being a set of (simple) roots. Given a rooted digraph $D=(V \cup R, A)$ and $r \in R$, an $r$-arborescence is called {\it spanning} if it spans $V\cup r$. If the $r$-arborescence spans all the vertices in $V$ reachable in $D$ from $r$, then it is called {\it reachability} $r$-arborescence.

A {\it (simply) matroid-rooted digraph}  is a tuple $(D,\mathcal{M})$ where {\boldmath $D$} $=(V\cup R,A)$ is a (simply) rooted digraph and {\boldmath $\mathcal{M}$} $=(R,r_{\mathcal{M}})$ is a matroid with ground set $R$ and rank function {\boldmath $r_{\mathcal{M}}$}.  

Given a matroid-rooted digraph $(D=(V \cup R, A),\mathcal{M}=(R,r_{\mathcal{M}}))$, we call an arborescence packing $\{B_r\}_{r\in R}$  {\it matroid-based} ({\it matroid-reachability-based}) if for all $r \in R$, the root of $B_r$ is $r$  and  for all $v \in V$, 
 the set of roots of the arborescences $B_r$ that contain $v$ forms a basis of $R$ (of $P_v^D\cap R$) in $\mathcal{M}$.

\subsection{Mixed hypergraphs}

We now turn our attention to mixed hypergraphs.
A {\it mixed hypergraph} is a tuple {\boldmath $\mathcal{H}$} $=(V,\mathcal{A}\cup \mathcal{E})$ where  {\boldmath $V$} is a  set of vertices, {\boldmath $\mathcal{A}$} is a set of directed hyperedges (dyperedges)  and {\boldmath $\mathcal{E}$} is a set  of hyperedges. A {\it dyperedge} $a$ is a tuple $(tail(a), head(a))$ where {\boldmath $head(a)$} is a single vertex in $V$ and {\boldmath $tail(a)$} is a nonempty subset of $V-head(a)$ and  a {\it hyperedge} is a subset of $V$ of size at least two. A mixed hypergraph without hyperedges is called a {\it directed hypergraph (dypergraph)}. 
A dyperedge $\vec e$ is an {\it orientation} of a hyperedge $e$ if $\vec e = (e-v,v)$ for some $v \in e$. A dypergraph $\vec{\mathcal{H}}=(V,\mathcal{A}\cup\vec{\mathcal{E}})$ is an orientation of $\mathcal{H}$ if $\vec{\mathcal{E}}$ is obtained from $\mathcal{E}$ by replacing $e$ by an orientation of $e$ for each $e\in\mathcal{E}.$
We say that $\mathcal{H}$ is a {\it mixed graph} if each dyperedge has a tail of size exactly one and each hyperedge has exactly two vertices.

Let $X \subseteq V.$ We say that a dyperedge $a\in \mathcal{A}$ {\it enters X}  if $head(a) \in X$ and $tail(a)-X\neq \emptyset$ and $a$ {\it leaves} $X$ if $a$ enters $V-X.$ We denote by {\boldmath $\rho^-_{\mathcal{A}}(X)$} the set of dyperedges entering $X$ and by {\boldmath $\rho^+_{\mathcal{A}}(X)$} the set of dyperedges leaving $X$. We use {\boldmath $d^-_{\mathcal{A}}(X)$} for $|\rho^-_{\mathcal{A}}(X)|$ and {\boldmath $d^+_{\mathcal{A}}(X)$} for $|\rho^+_{\mathcal{A}}(X)|$. We say that a hyperedge $e$ {\it enters} or {\it leaves} $X$ if $e$ intersects both $X$ and $V-X$ and denote by {\boldmath $d_{\mathcal{E}}(X)$} the number of hyperedges entering $X.$ We denote by {\boldmath$N_{\cal H}^+(X)$}   the set of heads of the dyperedges leaving $X$ and  by {\boldmath$N_{\cal H}^-(X)$}  the union of $tail(a)-X$ for all dyperedges $a\in {\cal A}$ entering $X.$ We call a vertex $r$ a {\it root} in $\mathcal{H}$ if $d_{\mathcal{A}}^-(r)=d_\mathcal{E}(r)=0$ and $tail(a)=\{r\}$ for all $a \in \rho_\mathcal{A}^+(r)$ and a {\it simple} root if additionally $d_\mathcal{A}^+(r)\le 1$. Given a subpartition $\{V_i\}_1^{\ell}$ of $V$, we denote by {\boldmath$e_{\mathcal{E}}(\{V_i\}_1^{\ell})$} the number of hyperedges in  $\mathcal{E}$ entering  some $V_i$ $(i \in \{1,\ldots,\ell\})$.

Making a {\it directed trimming} of a dyperedge $a$ means that $a$ is replaced by an arc $uv$ with $v=head(a)$ and $u\in tail(a)$.
Making a {\it directed trimming} of a hyperedge $e$ means that $e$ is replaced by an arc $uv$ for some distinct elements $u$ and $v$ in  $e$.
The mixed hypergraph $\mathcal{H}$ is called a {\it mixed hyperpath} ({\it mixed hyperarborescence}) if all the dyperedges and all the hyperedges can be made into a directed path (an arborescence) through directed trimming.  A {\it mixed $r$-hyperarborescence} for some $r \in V$ is a mixed hyperarborescence together with a vertex $r$  where the arborescence after the directed trimming can be chosen to be an $r$-arborescence. We also say that a digraph obtained by  making a directed trimming of a mixed hypergraph is a {\it directed trimming} of that mixed hypergraph.

For a vertex set $X \subseteq V$, we denote by {\boldmath $U_X^{\mathcal{H}}$} the set of vertices which are reachable from  at least one vertex in $X$ by a mixed hyperpath in $\mathcal{H}$, by {\boldmath$P^{\mathcal{H}}_X$} the set of vertices from which at least one vertex in $X$ is reachable by a mixed hyperpath in $\mathcal{H}$ and by {\boldmath $\mathcal{H}[X]$} the mixed hypergraph on  vertex set $X$ which contains all hyperedges and dyperedges of $\mathcal{H}$ which are completely contained in $X$. A {\it strongly connected component} of a mixed hypergraph is a maximal set of vertices that can be pairwise  reached from each other by a mixed hyperpath.    

We define a {\it (simply) rooted mixed hypergraph} as a mixed hypergraph $\mathcal{H}=(V\cup R,\mathcal{A}\cup \mathcal{E})$ with {\boldmath $R$} being a set of (simple) roots. Given a rooted mixed hypergraph ${\cal H}=(V \cup R, {\cal A}\cup {\cal E})$ and some $r \in R$, a mixed $r$-hyperarborescence is called {\it spanning} if it can be made into a spanning  $r$-arborescence by a directed trimming and  {\it reachability} if it can be made into an $r$-arborescence spanning all vertices in $U_r^{\mathcal{H}}$ by a directed trimming. For some $\mathcal{Z}\subseteq \mathcal{A}\cup \mathcal{E}$, we denote by {\boldmath $V_\mathcal{Z}$} all the vertices in $V$ which are contained in at least one hyperedge or dyperedge in $\mathcal{Z}$ and by {\boldmath $R_\mathcal{Z}$} all the vertices in $R$ which are contained in at least one dyperedge in $\mathcal{Z}$.
A {\it (simply) matroid-rooted mixed hypergraph} is a tuple $(\mathcal{H}, \mathcal{M})$ where $\mathcal{H}=(V\cup R,\mathcal{A}\cup \mathcal{E})$ is a (simply) rooted mixed hypergraph and {\boldmath $\mathcal{M}$} $=(R,r_{\mathcal{M}})$ is a matroid with ground set $R$ and rank function {\boldmath $r_{\mathcal{M}}$}.  

Given a matroid-rooted mixed hypergraph $({\cal H}=(V \cup R, {\cal A}\cup {\cal E}),\mathcal{M}=(R,r_{\mathcal{M}}))$, a mixed hyperarborescence packing $\{\mathcal{B}_r\}_{r \in R}$ is called {\it matroid-based} if every $\mathcal{B}_r$ can be made into an $r$-arborescence $B_r$ by a directed trimming such that $\{{B}_r\}_{r \in R}$ is a matroid-based  packing of arborescences. A mixed hyperarborescence packing  $\mathcal{B}=\{\mathcal{B}_r\}_{r \in R}$ is called {\it matroid-reachability-based} if every $\mathcal{B}_r$ can be made into an $r$-arborescence $B_r$ by a directed trimming such that for all $v \in V$, $\{r \in R:v \in V(B_r)\}$ is a basis of  $P_v^\mathcal{H}\cap R$ in $\mathcal{M}$. We use {\boldmath$\mathcal{A}(\mathcal{B})$} for the set of dyperedges contained in $\mathcal{B}$ and {\boldmath$\mathcal{E}(\mathcal{B})$} for the set of hyperedges contained in $\mathcal{B}$. Further, given a weight function $w:\mathcal{A} \cup \mathcal{E} \rightarrow \mathbb{R}$, we abbreviate $\sum_{a \in \mathcal{A}(\mathcal{B})}w(a)+\sum_{e \in \mathcal{E}(\mathcal{B})}w(e)$ to {\boldmath $w(\mathcal{B})$}. 

For a hyperedge $e \in \mathcal{E}$, its corresponding {\it bundle} {\boldmath $\mathcal{A}_e$} is the set of all possible orientations of $e$, i.e. $\mathcal{A}_e=\{(e-v,v):v \in e\}$. The {\it directed extension} {\boldmath$\mathcal{D}_{\mathcal{H}}$} $=(V \cup R, \mathcal{A}\cup \mathcal{A}_{\mathcal{E}})$ of $\mathcal{H}$ is obtained by replacing every $e \in \mathcal{E}$ by its corresponding bundle. Note that $\mathcal{A}_\mathcal{E}=\bigcup_{e\in\mathcal{E}} \mathcal{A}_e.$ We say that a packing of hyperarborescences in $\mathcal{D}_{\mathcal{H}}$ is {\it feasible} if it contains at most one dyperedge of the bundle $\mathcal{A}_e$ for all $e \in \mathcal{E}$.
\subsection{Bisets}

Finally, we  introduce some notation on bisets. 
Given a ground set $V$, a {\it biset} {\boldmath${\sf X}$} $=(X_O,X_I)$ consists of an {\it outer set} {\boldmath$X_O$} $\subseteq V$ and an {\it inner set} {\boldmath$X_I$} $\subseteq X_O$. We denote $X_O-X_I$, the {\it wall} of ${\sf X}$, by {\boldmath $X_W$}. 
For a subset $C$ of $V$, a collection of bisets $\{{\sf X}^i\}_1^{\ell}$ is called a {\it biset subpartition} of $C$ if $\{X_I^i\}_1^{\ell}$ is a subpartition of $C$ into nonempty sets and  $X_W^{i}\subseteq V-C$ for $i =1,\ldots, \ell$. In a dypergraph $\mathcal{D}=(V,\mathcal{A})$,  a dyperedge $a\in \mathcal{A}$ {\it enters} the biset ${\sf X}$  if $tail(a)-X_O\neq \emptyset$ and $head(a) \in X_I$. We use {\boldmath$\rho_{\mathcal{A}}^-({\sf X}$}$)$ for the set of dyperedges in ${\cal A}$ entering $X$ and {\boldmath $d_\mathcal{A}^-({\sf X})$} for $|\rho_{\mathcal{A}}^-({\sf X})|$.

\section{Results}\label{res}

This section introduces all the results considered and shows how our contributions relate to the previous results.

\subsection{Reachability in digraphs}

 The starting point of all studies on packing arborescences is the following theorem of Edmonds \cite{e} on packing spanning arborescences mentioned in a simpler form in the introduction.
 
\begin{theorem}(\cite{e})\label{basic}
Let $D=(V\cup R,A)$ be a simply rooted digraph. Then there exists a packing  $\{B_r\}_{r \in R}$ of  spanning $r$-arborescences  in $D$ if and only if for all $X\subseteq V\cup R$ with $X\cap V \neq \emptyset,$
\begin{equation}\label{edmonds}
d_A^{-}(X)\geq |R-X|.
\end{equation}
\end{theorem}

Theorem \ref{basic} has already been extended in \cite{e} by omitting the simplicity condition.

\begin{theorem}(\cite{e})\label{edstrong}
Let $D=(V\cup R,A)$ be a  rooted digraph. Then there exists a packing  $\{B_r\}_{r \in R}$ of  spanning $r$-arborescences  in $D$ 
if and only if  \eqref{edmonds} holds for all $X\subseteq V\cup R$ with $X\cap V\neq\emptyset$.
\end{theorem}

A simple proof of Theorem \ref{basic} due to Lov\'asz and a slightly more involved similar proof of Theorem \ref{edstrong} due to Frank can be found in \cite{book}. Nevertheless, no way to conclude Theorem \ref{edstrong} from  Theorem \ref{basic} is known.
\medskip

We now turn our attention to packing reachability arborescences.
The following generalization of Theorem \ref{edstrong} is equivalent to the fundamental result of Kamiyama, Katoh and Takizawa \cite{kkt}. In \cite{kkt}, the simply rooted version is presented, however the rooted version can be easily reduced to the simply rooted version. Indeed, by adding for every root $r$, a new  vertex $r'$ and an arc $r'r$, we obtain a simply rooted instance of the problem which is equivalent to the original one.

\begin{theorem}(\cite{kkt}, \cite{book})\label{3japt}
Let $D=(V\cup R,A)$ be a rooted digraph. Then there exists a packing $\{B_r\}_{r \in R}$ of reachability $r$-arborescences   in $D$ if and only if for all $X\subseteq V\cup R$ with $X\cap V\neq\emptyset,$
\begin{equation}\label{3jap}
d_A^{-}(X)\geq |P_X^D\cap R|-|X\cap R|.
\end{equation}
\end{theorem}

Our first contribution is to show that surprisingly Theorem \ref{edstrong} implies Theorem \ref{3japt}. The very simple inductive proof can be found in Section \ref{red}.

\subsection{Reachability in matroid-rooted digraphs}

We now present generalizations of the concepts above, namely matroid-based packings and matroid-reachability-based packings.  

The first result on matroid-based packings of arborescences is due to Durand de Gevigney, Nguyen and Szigeti \cite{dns}. It is equivalent to the following theorem (see \cite{surveyZ}).

\begin{theorem}(\cite{dns})\label{mat}
Let $(D=(V\cup R,A),\mathcal{M}=(R, r_{\mathcal{M}}))$ be a simply matroid-rooted digraph. Then there exists a matroid-based packing of arborescences in $(D,\mathcal{M})$ if and only if for all  $X \subseteq V\cup R$ with $X\cap V\neq\emptyset$ and $X \cap R=N_D^-(X \cap V)\cap R$,
\begin{equation}\label{matcond}
d_A^{-}(X)\geq r_{\mathcal{M}}(R)-r_{\mathcal{M}}(X \cap R).
\end{equation}
\end{theorem}


We  mention that the simplicity condition in Theorem \ref{mat} can be omitted. This result might also be interesting for itself. It plays the same role for matroid-based packings as Theorem \ref{edstrong} played for spanning arborescence packings. 
While it is not known whether Theorem \ref{basic} implies Theorem \ref{edstrong}, 
the stronger matroid setting allows to directly derive the following theorem  from Theorem \ref{mat}.

\begin{theorem}\label{matbranch}
Let $(D=(V\cup R,A),\mathcal{M}=(R, r_{\mathcal{M}}))$ be a matroid-rooted digraph. Then there exists a matroid-based packing of arborescences in $(D,\mathcal{M})$ if and only if \eqref{matcond} holds for all  $X \subseteq V\cup R$ with $X\cap V\neq\emptyset$ and $X \cap R=N_D^-(X \cap V)\cap R$.
\end{theorem}

It is easy to see that Theorem \ref{matbranch} follows from Theorem \ref{mat}. Indeed, we may define a simply matroid-rooted digraph $(D',\mathcal{M}')$ obtained from  $(D,\mathcal{M})$ by replacing every root $r \in R$ by a set $Q_r$ of  $|N_D^+(r)|$  simple roots in the digraph such that $N_{D'}^+(Q_r)=N_D^+(r)$ and by $|Q_r|$ parallel copies of  $r$ in the matroid. Since $(D,\mathcal{M})$ satisfies \eqref{matcond}, so does $(D',\mathcal{M}')$. By Theorem \ref{mat}, there exists  a matroid-based packing of arborescences in $(D',\mathcal{M}')$, from which we obtain, by identifying all vertices of $Q_r$ into $r$ for all $r\in R$,  a matroid-based packing of arborescences  in $(D,\mathcal{M})$.
\medskip

A reachability extension of Theorem \ref{mat} was obtained by Kir\'aly \cite{k}. We deduce  the following slightly stronger version of it from Theorem \ref{matbranch}   in Section \ref{red}.

\begin{theorem}\label{matreach}
Let $(D=(V\cup R,A),\mathcal{M}=(R, r_{\mathcal{M}}))$ be a matroid-rooted digraph. Then there exists a matroid-reachability-based packing of arborescences in $(D,\mathcal{M})$ if and only if for all  $X \subseteq V\cup R$ with $X\cap V\neq \emptyset$ and $X \cap R=N_D^-(X \cap V)\cap R$,
\begin{equation}\label{matreachcond}
d_A^{-}(X)\geq r_{\mathcal{M}}(P_X^D\cap R)-r_{\mathcal{M}}(X \cap R).
\end{equation}
\end{theorem}

\subsection{Generalizations to mixed hypergraphs}

This part deals with another way of generalizing Theorem \ref{basic}: rather than changing the requirements on the packing, one can consider changing the basic objects of consideration from digraphs to more general objects. 
One such generalization was dealt with by Frank, Kir\'aly and Kir\'aly \cite{fkk}. They considered dypergraphs instead of digraphs and they generalized Theorem \ref{basic} to dypergraphs. A  result where the concepts of reachability and dypergraphs were combined was obtained by B\'erczi and Frank in \cite{bf}.
Yet another class Theorem \ref{basic} can be generalized to was considered by Frank in \cite{f}: mixed graphs. 
He gave a characterization of mixed graphs admitting a mixed spanning arborescence packing. 

A natural question now is whether several of the aforementioned generalizations can be combined into a single one. In \cite{fklst}, the authors surveyed all possible combinations of these generalizations and gave an overview of all existing results.  A significant amount of cases was covered  by Fortier et al \cite{fklst}. They first prove a characterization combining the concepts of dypergraphs, matroids and reachability. They further prove a theorem that combines the concepts of matroids, hypergraphs and mixed graphs. This result can be stated as follows in our framework. 
We use it in Section \ref{red} for one of the reductions.

\begin{theorem}\label{matmixhyp}(\cite{fklst})
Let $(\mathcal{H}=(V\cup R,\mathcal{A}\cup \mathcal{E}), \mathcal{M}=(R, r_{\mathcal{M}}))$ be a simply matroid-rooted mixed hypergraph. Then there exists a matroid-based  packing of mixed hyperarborescences in $(\mathcal{H}, \mathcal{M})$  if and only if for every biset subpartition $\{{\sf X}^i\}_1^\ell$ of $V$ with $X^{i}_W=N_\mathcal{H}^-(X_I^i)\cap R$ for $i=1,\ldots,\ell$,
\begin{equation}\label{condmatmixhyp}
e_{\mathcal{E}}(\{X_I^i\}_1^{\ell})+\sum_{i=1}^\ell d^-_{\mathcal{A}}({\sf X}^{i})\geq \sum_{i=1}^\ell(r_{\mathcal{M}}(R)-r_{\mathcal{M}}(X_W^{i})).
\end{equation}
\end{theorem}

\subsection{Reachability in mixed graphs}

Theorem \ref{matmixhyp} had a lot of corollaries generalizing Theorem \ref{basic}, however, the cases of combinations including both reachability and mixed graphs remained open. They seemed hard to deal with as all natural generalizations failed. Indeed, it turned out that the remaining cases required a deeper concept, namely the use of bisets. While the use of bisets in our statement of Theorem \ref{matmixhyp} is only for convenience, it is essential in the following theorems. 
\medskip

The following theorem is equivalent to the result of  Matsuoka and Tanigawa \cite{mt} on reachability mixed arborescence packing, as it was shown in  \cite{gy}. 

\begin{theorem}(\cite{mt})\label{mixreach}
Let $H=(V\cup R,A\cup E)$ be a rooted mixed graph. Then there exists a packing of reachability mixed $r$-arborescences  $\{B_r\}_{r \in R}$ in $H$ if and only if for every biset subpartition $\{{\sf X}^i\}_1^\ell$ of a strongly connected component $C$ of $H-R$ such that $X_W^i=P^H_{X_W^i}$ 
 for all $i=1,\ldots,\ell,$
\begin{equation*}
e_{E}(\{X_I^i\}_1^{\ell})+\sum_{i=1}^\ell d^-_{A}({\sf X}^{i})\geq \sum_{i=1}^\ell (|P^H_C\cap R|-|X_W^{i}\cap R|).
\end{equation*}
\end{theorem}

In \cite{mt}, Matsuoka and Tanigawa suggest a possible generalization of Theorem \ref{mixreach} to the matroidal case. Such a generalization was proven by Gao and Yang \cite{gy} using a slightly different method than the one suggested in \cite{mt}.
They proved the following result.

\begin{theorem}(\cite{gy})\label{matmixreach}
Let $(H=(V\cup R,A\cup E),\mathcal{M}=(R,r_{\mathcal{M}}))$ be a matroid-rooted mixed graph. Then there exists a matroid-reachability-based packing of mixed arborescences  in $(H,\mathcal{M})$ if and only if for every biset subpartition $\{{\sf X}^i\}_1^\ell$ of a strongly connected component $C$ of $H-R$ such that $X_W^i=P^H_{X_W^i}$ for all $i=1,\ldots,\ell,$
\begin{equation*}
e_{E}(\{X_I^i\}_1^{\ell})+\sum_{i=1}^\ell d^-_{A}({\sf X}^{i})\geq \sum_{i=1}^\ell (r_{\mathcal{M}}(P^H_C\cap R)-r_{\mathcal{M}}(X_W^{i}\cap R)).
\end{equation*}
\end{theorem}

\subsection{New results}

The remaining open problems were the generalizations of Theorems \ref{mixreach} and \ref{matmixreach} to mixed hypergraphs. Proving such generalizations is another contribution of this article. While such a result can be obtained by the proof technique used by Gao and Yang \cite{gy} for Theorem \ref{matmixreach}, we follow a different approach: we derive such a characterization from Theorem \ref{matmixhyp}. Again, we first show that the simplicity condition in Theorem \ref{matmixhyp} can be omitted.

\begin{theorem}\label{matmixhypbranch}
Let $(\mathcal{H}=(V\cup R,\mathcal{A}\cup \mathcal{E}), \mathcal{M}=(R, r_{\mathcal{M}}))$ be a matroid-rooted mixed hypergraph. Then there exists a matroid-based packing of mixed hyperarborescences  in $(\mathcal{H}, \mathcal{M})$ if and only if \eqref{condmatmixhyp} holds for every biset subpartition $\{{\sf X}^i\}_1^\ell$ of $V$ with 
$X_W^i=N_\mathcal{H}^-(X^{i}_I)\cap R$ for $i=1,\ldots,\ell$.
\end{theorem}

Theorem \ref{matmixhypbranch}  allows us to derive the following new theorem. Observe that this is a common generalization of all the theorems mentioned before in this article. It includes all the theorems surveyed in \cite{fklst}. 

\begin{theorem}\label{matmixhypreach}
Let $(\mathcal{H}=(V\cup R, \mathcal{A}\cup \mathcal{E}),\mathcal{M}=(R,r_{\mathcal{M}}))$ be a matroid-rooted mixed hypergraph. Then there exists a matroid-reachability-based packing of mixed hyperarborescences  in $(\mathcal{H}, \mathcal{M})$ if and only if for every biset subpartition $\{{\sf X}^i\}_1^\ell$ of a strongly connected component $C$ of $\mathcal{H}-R$  such that $X_W^i=P^{\mathcal{H}}_{X_W^i}$ for all $i=1,\ldots,\ell,$
\begin{equation}\label{4cond}
e_{\mathcal{E}}(\{X_I^i\}_1^{\ell})+\sum_{i=1}^\ell d^{-}_{\mathcal{A}}({\sf X}^{i})\geq \sum_{i=1}^\ell(r_{\mathcal{M}}(P^{\mathcal{H}}_C\cap R)-r_{\mathcal{M}}(X^{i}_W\cap R)).
\end{equation}
\end{theorem}

We obtain the only remaining case, a generalization of Theorem \ref{mixreach} to mixed hypergraphs as a corollary  by applying Theorem \ref{matmixhypreach} to the free matroid.

\begin{corollary}\label{mixhypreach}
Let $\mathcal{H}=(V\cup R, \mathcal{A}\cup \mathcal{E})$ be a rooted mixed hypergraph. Then there exists a packing of reachability mixed $r$-hyperarborescences  $\{\mathcal{B}_r\}_{r \in R}$ in $\mathcal{H}$ if and only if for every biset subpartition $\{{\sf X}^i\}_1^\ell$ of a strongly connected component $C$ of $\mathcal{H}-R$  such that $X_W^i=P^{\mathcal{H}}_{X_W^i}$ for all $i=1,\ldots,\ell,$
\begin{equation*}\label{4corcond}
e_{\mathcal{E}}(\{X_I^i\}_1^{\ell})+\sum_{i=1}^\ell d^{-}_{\mathcal{A}}({\sf X}^{i})\geq \sum_{i=1}^\ell(|P^{\mathcal{H}}_C\cap R|-|X^{i}_W\cap R|).
\end{equation*}
\end{corollary}

\section{Reductions}\label{red}

This section contains the proofs of the old and new theorems that we mentioned before. All the proofs work by reductions from the spanning versions to the reachability versions.

\subsection{Proof of Theorem \ref{3japt}}\label{lfbdjhfvhd}

The proof uses Theorem \ref{edstrong} and is self-contained otherwise. 

\begin{proof}(of {\bf Theorem \ref{3japt}})
Necessity is evident. For sufficiency, let {\boldmath $D=(V\cup R,A)$} be a minimum counterexample with respect to $|V|$. Obviously, $V \neq \emptyset$.  Let {\boldmath $C$} $\subseteq V$ be the vertex set of a strongly connected component of $D$ that has no arc leaving. Since each $r\in R$ is a root, $C$ exists. Note that each vertex of $C$ is reachable in $D$ from the same set {\boldmath $R_2$} of roots in $R$ since $D[C]$ is strongly connected. We can hence divide  the problem into two subproblems, a smaller one on reachability arborescence packing  and one on spanning arborescence packing.

Let the rooted digraph {\boldmath $D_1$} be equal to $D-C.$
Since no arc leaves $C$ in $D$ and $D$ satisfies \eqref{3jap}, so does $D_1.$ Hence, by the minimality of $D$, $D_1$ has a reachability arborescence packing $\{${\boldmath $B^1_r$}$\}_{r\in R}$.

We now create a second rooted digraph {\boldmath $D_2$} $=(V_2\cup R_2,A_2)$. Let {\boldmath$V_2$} $=C\cup T$ where {\boldmath $T$} $=\{${\boldmath $t_{uv}$}$:uv \in \rho_A^{-}(C)\}$ is a set of new vertices.  Let {\boldmath $A_2$} contain $A(D[C])$, $\{rt_{uv}:r\in R_2, u\in U^D_r, t_{uv} \in T\}$ and for all $t_{uv} \in T,$  one arc from $t_{uv}$ to $v$ and $|R_2|$ arcs from $v$ to $t_{uv}$.

\begin{lemma}\label{d''pack}
$D_2$ satisfies \eqref{edmonds}.
\end{lemma}

\begin{proof}
Let $X \subseteq V_2\cup R_2$ with $X\cap V_2\neq \emptyset$. We must show  that  $d_{A_2}^{-}(X)\geq |R_2-X|$.
If $X \cap C= \emptyset$, then, by $X\cap V_2\neq\emptyset$,  there exists  some $t_{uv} \in X\cap T$. It follows that 
 $d_{A_2}^{-}(X)\geq d_{A_2}(v,t_{uv})=|R_2|\ge |R_2-X|$. If $X \cap C\neq\emptyset,$ then, since $D[C]$ is strongly connected, $R_2=P^D_C\cap R=P^D_{X\cap C}\cap R$. Let {\boldmath $Y$} $=(V\cup R)-{U}^D_{R_2-X},$  {\boldmath $Z$} $= (X \cap C)\cup Y$ and {\boldmath $uv$} $\in \rho^{-}_A(Z)$. Since $u\notin Y$, we get $u\in {U}^D_{R_2-X},$ so  $u \in U^{D}_{\bar r}$ for some {\boldmath ${\bar r}$} $\in R_2-X$. Since $\rho^-_A(Y)=\emptyset$, we have $v \in X \cap C$.  If $u \in C$, then $u\in C-Z\subseteq C-X,$ so $uv \in \rho_{A_2}^{-}(X)$. Otherwise, $t_{uv}\in T$, so ${\bar r}t_{uv}, t_{uv}v\in A_2$. Since $v\in X$ and ${\bar r}\notin X,$ ${\bar r}t_{uv}$ or $t_{uv}v\in\rho_{A_2}^-(X)$. Note that for distinct arcs in $A$ entering $Z$ distinct arcs in $A_2$ entering $X$ have been found above. Thus, by \eqref{3jap},  $d_{A_2}^{-}(X)\geq  d_A^{-}(Z)\geq |P^D_Z\cap R|-|Z\cap R|=|(P^D_{X\cap C}\cup P^D_Y)\cap R|-|((X\cap C)\cup Y)\cap R|=|R|-|R-(R_2-X)|=|R_2-X|.$
\end{proof}

By Lemma \ref{d''pack} and Theorem \ref{edstrong}, $D_2$ has a spanning arborescence packing  $\{${\boldmath $B^2_r$}$\}_{r\in  R_2}$. 

With the help of the packings $\{B^1_r\}_{r \in R}$ in $D_1$ and $\{B^2_r\}_{r \in R_2}$ in $D_2$, we construct a reachability arborescence packing in $D$ yielding a contradiction.
%
%
For all $r\in R-R_2,$ let {\boldmath $B_r$} $=B^1_r$  and  for all $r\in R_2$, let {\boldmath $B_r$} be obtained from the union of $B_r^1$ and $B^2_r-(R_2\cup T)$  by adding the arc $uv$ for all $t_{uv}v \in A(B^2_r)$. It is easy to see that $\{B_r\}_{r \in R}$ is a reachability arborescence packing in $D$ which yields a contradiction.
 \end{proof}
 
\subsection{Proof of Theorem \ref{matreach}}

In this section, 
we show how to derive Theorem \ref{matreach} from Theorem \ref{matbranch}. The role of Theorem \ref{matbranch} in the proof is similar to the role of Theorem \ref{edstrong} in the proof of Theorem \ref{3japt}. While the proof contains similar ideas to the ones in the proof of Theorem \ref{3japt}, it is somewhat more technical.

\begin{proof}(of {\bf Theorem \ref{matreach}}) Necessity is evident. 
For sufficiency, let {\boldmath$(D=(V \cup R, A), \mathcal{M}=(R,r_{\mathcal{M}}))$} be a minimum counterexample with respect to $|V|$. Obviously $V \neq \emptyset$. 
Let {\boldmath $C$} $\subseteq V$ be the vertex set of a strongly connected component of $D$ that has no arc leaving. 
Since each $r\in R$ is a root, $C$ exists.

Let {\boldmath $D_1=(V_1\cup R,A_1)$} $=D-C.$ Note that $(D_1,\mathcal{M})$ is a matroid-rooted digraph.
By $d_A^+(C)=0$, we have $d_{A_1}^{-}(X)=d_{A}^{-}(X)$ and 
 $P^{D_1}_X=P^D_X \text{ for all } X\subseteq V_1\cup R.$ 
Then, since $(D,\mathcal{M})$ satisfies \eqref{matreachcond}, so does $(D_1,\mathcal{M}).$ Hence, by the minimality of $D$, $(D_1,\mathcal{M})$ has a matroid-reachability-based packing of arborescences  $\{${\boldmath $B^1_r$}$\}_{r \in R}.$ 

We now define  a matroid-rooted digraph $(D_2,\mathcal{M}_2)$ which depends on the arborescences  $\{B^1_r\}_{r \in R}.$ Let {\boldmath $R_2$} $=P^D_C\cap R$, {\boldmath $\mathcal{M}_2$} $=\mathcal{M}|R_2$ and {\boldmath $D_2$} $=(V_2\cup R_2,A_2)$ where {\boldmath $V_2$} $=C\cup T$ for a set {\boldmath $T$} $=\{ ${\boldmath $t_{uv}$} $:uv \in \rho_A^{-}(C)\}$ of new vertices and  {\boldmath $A_2$} contains $A(D[C]),$  $\{rt_{uv}:r\in R_2, u\in V(B^1_r), t_{uv} \in T\}$ and  for every $t_{uv} \in T$, one arc from $t_{uv}$ to $v$ and $r_{\mathcal{M}_2}(R_2)$ arcs from $v$ to $t_{uv}$.

\begin{lemma}\label{reachd''pack}
$(D_2,\mathcal{M}_2)$ satisfies \eqref{matcond}.
\end{lemma}

\begin{proof}
Let {\boldmath $X$}  $\subseteq V_2\cup R_2$ with $X\cap V_2\neq\emptyset$ and $X\cap R_2=N_{D_2}^-(X\cap V_2)\cap R_2$.

If $X \cap C= \emptyset$, then, by $X\cap V_2\neq\emptyset$,  there exists  some $t_{uv} \in X\cap T$. It follows that  $d_{A_2}^{-}(X)\geq d_{A_2}(v,t_{uv})=r_{\mathcal{M}_2}(R_2)\ge r_{\mathcal{M}_2}(R_2)-r_{\mathcal{M}_2}(X\cap R_2)$, so \eqref{matcond} holds for $X$ in $(D_2,\mathcal{M}_2)$.

If $X\cap C \neq \emptyset$,  then, since $D[C]$ is strongly connected, we have $R_2=P^D_C\cap R=P^D_{X\cap C}\cap R.$
Let {\boldmath $R'$} $=span_\mathcal{M}(X\cap R_2),$ {\boldmath $Y$} $=(V\cup R)-{U}^D_{R-R'}$ and {\boldmath $Z$} $=(X \cap C)\cup Y.$ Then we have
\begin{eqnarray}
&&r_\mathcal{M}(P^D_Z\cap R)=r_\mathcal{M}(R_2\cup R')=r_\mathcal{M}(R_2)=r_{\mathcal{M}_2}(R_2),\label{nkjffkvn}\\
&&r_\mathcal{M}(Z \cap R)=r_\mathcal{M}(R')= r_\mathcal{M}(X \cap R_2)= r_{\mathcal{M}_2}(X \cap R_2).\label{lkdijbiu}
\end{eqnarray}

\begin{claim}\label{conddxdy}
 $d_{A_2}^{-}(X)\geq d_{A}^{-}(Z)$. 
 \end{claim}

\begin{proof}
 Let {\boldmath $uv$} $\in \rho_A^-(Z)$. Since $u\notin Y$, we get $u\in {U}^D_{R-R'},$ so  $u \in U^{D}_{\bar r}$ for some {\boldmath ${\bar r}$} $\in R-R'$. Since $\rho^-_A(Y)=\emptyset$, we have $v \in X \cap C$.  If $u \in C$, then $u\in C-Z\subseteq C-X,$ so $uv \in \rho_{A_2}^{-}(X)$. Otherwise, $uv$ enters $C$ in $D$, so $t_{uv}\in T$. If $t_{uv} \in X$, then, since $\{r \in R: u \in V(B^1_r)\}$ is a basis of $P_u^{D}\cap R$ in $\mathcal{M}$, we have
${\bar r}\notin R'= span_\mathcal{M}(X\cap R_2)=span_{\mathcal{M}}(N_{D_2}^-(X\cap V_2)\cap R_2)\supseteq span_{\mathcal{M}}(N_{D_2}^-(t_{uv})\cap R_2)=span_{\mathcal{M}}(\{r \in R_2:u \in V(B^1_r)\})=span_{\mathcal{M}}(\{r \in R:u \in V(B^1_r)\})\supseteq P_u^{D}\cap R\supseteq \{{\bar r}\},$ 
 a contradiction. 
Thus $t_{uv}\notin X$ and so $t_{uv}v \in \rho_{A_2}^{-}(X)$.  Note that for distinct arcs in $A$ entering $Z$ distinct arcs in $A_2$ entering $X$ have been found above, so the claim follows.
\end{proof}

%
%

\begin{claim}\label{condZ}
 The set $Z$ satisfies \eqref{matreachcond}. 
 \end{claim}
 
\begin{proof}Observe that the claim does not directly follow from the assumption as $Z$ does not necessarily satisfy the condition $Z \cap R=N_D^-(Z \cap V) \cap R$. We now define {\boldmath$Z'$} $=(Z \cap V)\cup (N_D^-(Z \cap V) \cap R)$. Then $(Z-Z')\cup(Z'-Z)\subseteq R.$
Note that $Z' \cap R=N_D^-(Z' \cap V) \cap R$, hence $Z'$ satisfies \eqref{matreachcond} by assumption. Next observe that   no arc exists from $Z-Z'$ to $Z \cap V$ and for every $r \in Z'\cap R$,  at least one arc exists from $r$ to $Z \cap V$. This yields $d_A^-(Z)\ge d_A^-(Z')+|Z'-Z|$. Finally, we obtain $(P_{Z}^D\cap R) -(Z- Z')=P_{Z\cap V}^D\cap R
\subseteq P_{Z'}^D\cap R$, $((P_{Z}^D\cap R) -(Z- Z'))\cup (Z\cap R)=P_{Z}^D\cap R$ and $((P_{Z}^D\cap R) -(Z- Z'))\cap (Z\cap R)=Z \cap Z' \cap R$.  
Hence, since $r_{\mathcal{M}}$ is subcardinal, non-decreasing and submodular, we obtain
\begin{align*}
d_A^-(Z)&\ge d_A^-(Z')+|Z'-Z|\\
&\ge r_{\mathcal{M}}(P_{Z'}^D\cap R)-r_{\mathcal{M}}({Z'}\cap R)+r_{\mathcal{M}}(Z'-Z)\\
&\ge r_{\mathcal{M}}((P_{Z}^D\cap R) -(Z- Z'))-r_{\mathcal{M}}(Z \cap Z' \cap R)\\
&\ge r_{\mathcal{M}}(P_{Z}^D\cap R)-r_{\mathcal{M}}(Z\cap R),
\end{align*}
and the claim follows.
 \end{proof}

By Claims \ref{conddxdy} and \ref{condZ}, \eqref{nkjffkvn} and \eqref{lkdijbiu}, 
we obtain 
$d_{A_2}^{-}(X)\geq d_{A}^{-}(Z)\geq  r_{\mathcal{M}}(P^D_Z\cap R)-r_{\mathcal{M}}(Z\cap R)=r_{\mathcal{M}_2}(R_2)-r_{\mathcal{M}_2}(X \cap R_2),$
so \eqref{matcond} holds for $X$ in $(D_2,\mathcal{M}_2)$. This completes the proof of  Lemma \ref{reachd''pack}.
\end{proof}

By Lemma \ref{reachd''pack} and Theorem \ref{matbranch}, $(D_2,\mathcal{M}_2)$ has a matroid-based packing of arborescences $\{${\boldmath $B^2_r$}$\}_{r \in R_2}.$

 We finally give a construction of a  packing of the desired form in $(D,\mathcal{M})$ using the packings $\{B^1_r\}_{r \in R}$ and $\{B^2_r\}_{r \in R_2}$, which yields a contradiction.

\begin{lemma}\label{reachdpack}
$(D,\mathcal{M})$ has a matroid-reachability-based packing of arborescences.
\end{lemma}

\begin{proof}
For all $r\in R-R_2,$ let {\boldmath $B_r$} $=B^1_r$  and  for all $r\in R_2$, let {\boldmath $B_r$} be obtained from the union of $B_r^1$ and $B^2_r-(R_2\cup T)$  by adding the arc $uv$ for all $t_{uv}v \in A(B^2_r)$.

We first show that $\{B_r\}_{r \in R}$ is a packing of $r$-arborescences. Since $\{B^1_r\}_{r \in R}$ and $\{B^2_r\}_{r \in R_2}$ are packings, so is $\{B_r\}_{r \in R}$. For $r\in R-R_2$, $B_r=B_r^1$ is an $r$-arborescence. 
Now let  {\boldmath$r$} $\in R_2$
and {\boldmath$v$} $\in V(B_r)-r.$ If $v\notin C$, then $v\in U_r^{B_r^1}\subseteq U_r^{B_r}$ and $d_{A(B_r)}^-(v)=d_{A(B^1_r)}^-(v)=1.$ Otherwise, $v\in U_r^{B_r^2}$ and $d_{A(B^2_r)}^-(v)=1$ and when $t_{uv}v\in A(B^2_r)$ is replaced by $uv\in A(B_r)$ then $u\in V(B_r^1)$. Hence we have $v\in U_r^{B_r}$ and $d_{A(B_r)}^-(v)=1.$ Since $r$ is a root,
it follows that $B_r$ is an $r$-arborescence.

Finally, we show that the packing $\{B_r\}_{r \in R}$ is  matroid-reachability-based. For $v \in V-C$, we have  $\{r \in R:v \in V(B_r)\}=\{r \in R:v \in V(B^1_r)\}$ which is a basis of $P_v^{D_1}\cap R=P_v^D\cap R$ in $\mathcal{M}$ by $d_A^+(C)=0$. For $v \in C$, we have  $\{r \in R:v \in V(B_r)\}=\{r \in R_2:v \in V(B^2_r)\}$ which is a basis of $R_2$ in  $\mathcal{M}_2$, so a basis of $R_2=P_v^D\cap R$ in $\mathcal{M}$. 
\end{proof}

Lemma \ref{reachdpack} contradicts the fact that $(D,\mathcal{M})$ is a counterexample and hence completes the proof of Theorem \ref{matreach}.
\end{proof}

\subsection{Proof of Theorems \ref{matmixhypbranch} and \ref{matmixhypreach}}

We first derive Theorem \ref{matmixhypbranch} from Theorem \ref{matmixhyp}. 

\begin{proof}(of {\bf Theorem \ref{matmixhypbranch}})
Necessity is evident. 
For sufficiency, we define a simply matroid-rooted mixed hypergraph {\boldmath $(\mathcal{H}'=(V\cup R',\mathcal{A}'\cup \mathcal{E}),{\mathcal{M}}'=(R',r_{\mathcal{M}'}))$} obtained from  $(\mathcal{H},\mathcal{M})$ by replacing every root $r \in R$ by a set {\boldmath $Q_r$} of  $|N_\mathcal{H}^+(r)|$  simple roots such that $N_{\mathcal{H}'}^+(Q_r)=N_\mathcal{H}^+(r)$ in the mixed hypergraph and by $|Q_r|$ parallel copies of  $r$ in the matroid.

Now let {\boldmath $\{{\sf X}^i\}_1^\ell$} be a biset subpartition of $V$ with $X_W^i=N_{\mathcal{H}'}^-(X^{i}_I)\cap R'$ for $i=1,\ldots,\ell$. Let {\boldmath $i$} $\in \{1,\ldots,\ell\}$.  Let the biset {\boldmath ${\sf Y}^{i}$} be defined as $(X^{i}_I \cup \{r\in R:Q_r\cap X_W^i\neq\emptyset\},X_I^{i})$. Observe that $Y_I^i=X_I^i$, $Y_W^i=N_\mathcal{H}^-(Y^{i}_I)\cap R,$ $d_{\mathcal{A}}^-({\sf Y}^{i})\leq d_{\mathcal{A}'}^-({\sf X}^{i})$, $r_{\mathcal{M}}(R)=r_{\mathcal{M}'}(R')$  and $r_{\mathcal{M}}(Y_W^i)=r_{\mathcal{M}'}(X_W^i)$. Then, by \eqref{condmatmixhyp} applied for $\{{\sf Y}^{i}\}_{i=1}^{\ell}$, we obtain 
 $e_{\mathcal{E}}(\{X_I^i\}_1^{\ell})= e_{\mathcal{E}}(\{Y_I^i\}_1^{\ell})
 					\geq\sum_{i=1}^\ell( r_{\mathcal{M}}(R)-r_{\mathcal{M}}(Y_W^i)-d_{\mathcal{A}}^{-}({\sf Y}^{i}))
 					\geq \sum_{i=1}^\ell( r_{\mathcal{M}'}(R')-r_{\mathcal{M}'}(X_W^i)-d_{\mathcal{A}'}^{-}({\sf X}^{i})),$
  that is $(\mathcal{H}',\mathcal{M}')$ satisfies \eqref{condmatmixhyp}.

We may hence apply Theorem \ref{matmixhyp}  to obtain  in $(\mathcal{H}',\mathcal{M}')$ a matroid-based packing of mixed hyperarborescences  $\{${\boldmath$\mathcal{B}'_{r'}$}$\}_{r' \in R'}$ with arborescences $\{${\boldmath${B}'_{r'}$}$\}_{r' \in R'}$ as  directed trimmings like in the definition. For all $r \in R$, let {\boldmath$\mathcal{B}_r$} and {\boldmath$B_r$} be obtained from  $\{\mathcal{B}'_{r'}\}_{{r'}\in Q_r}$ and $\{B'_{r'}\}_{{r'}\in Q_r}$ by identifying all vertices of $Q_r$ into $r$. Note that $B_r$ is a  directed trimming of $\mathcal{B}_r$ for all $r \in R$. To show that $\{\mathcal{B}_r\}_{r \in R}$ is a packing of mixed hyperarborescences with the desired properties it is enough to show that $\{B_r\}_{r \in R}$ is a matroid-based packing of arborescences. Since $\{B'_{r'}\}_{r' \in R'}$ is a packing, so is $\{B_r\}_{r \in R}$. Let $r\in R$. Since $\{r' \in R':v \in V(B'_{r'})\}$ is independent in $\mathcal{M}'$ for all $v \in V$ and $Q_r$ is a set of parallel elements in $\mathcal{M}'$, $\{B'_{r'}\}_{r' \in Q_r}$ is a set of vertex-disjoint $r'$-arborescences  and hence $B_r$ is an $r$-arborescence. Moreover, for all $v\in V$, we have that $\{r \in R:v \in V(B_r)\}$ is a basis of $\mathcal{M}$ because $\{r' \in R':v \in V(B'_{r'})\}$ is a basis of $\mathcal{M}'$.
\end{proof}

We are now ready to derive Theorem \ref{matmixhypreach} from Theorem \ref{matmixhypbranch}. Again, the proof has certain similarities to the previous ones.

\begin{proof}(of {\bf Theorem \ref{matmixhypreach}})
We first prove {\bf necessity}. Suppose that there exists a matroid-reachability-based  packing of mixed hyperarborescences $\{${\boldmath $\mathcal{B}_r$}$\}_{r \in R}$. By definition, for every $r \in R$, there is an $r$-arborescence $B_r$ that is a  directed trimming of $\mathcal{B}_r$ with $\{r \in R:v\in V(B_r)\}$ being a basis of $P^{\mathcal{H}}_v\cap R$ in $\mathcal{M}$ for all $v \in V$. Let {\boldmath $\{{\sf X}^i\}_1^\ell$} be a biset subpartition of a strongly connected component {\boldmath $C$} of $\mathcal{H}-R$ such that $X_W^i=P^{\mathcal{H}}_{X_W^i}$ for all $i=1,\ldots,\ell$. 
Let {\boldmath $i$} $\in\{1,\ldots,\ell\}$, {\boldmath $R_i$} $=\{r \in R-X_W^{i}:V(B_r)\cap X_I^{i}\neq \emptyset\}$ and {\boldmath $v$} $\in X_I^i.$
 Then we have $r_{\mathcal{M}}(R_i\cup (X^{i}_W\cap R))\ge r_{\mathcal{M}}(\{r\in R:v\in V(B_r)\})=r_{\mathcal{M}}(P^{\mathcal{H}}_v\cap R)=r_{\mathcal{M}}(P^{\mathcal{H}}_C\cap R).$
Thus, by the subcardinality and the submodularity of $r_{\mathcal{M}}$, we can bound the size of $R_i$ from below:
$|R_i|\geq r_{\mathcal{M}}(R_i)\geq r_{\mathcal{M}}(R_i\cup (X^{i}_W\cap R))-r_{\mathcal{M}}(X^{i}_W\cap R)\ge  r_{\mathcal{M}}(P^{\mathcal{H}}_C\cap R)-r_{\mathcal{M}}(X^{i}_W\cap R).$
%
Since $X_W^i=P^{\mathcal{H}}_{X_W^i}$, no dyperedge and no hyperedge enters  $X_W^i$ in $\mathcal{H}.$ Then, by $v\in X^i_I,$ every $B_r$ with $r \in R_i$ has an arc that enters ${\sf X}^{i}$, that is  $\mathcal{B}_r$  contains either a dyperedge in ${\mathcal{A}}$ entering ${\sf X}^{i}$ or a hyperedge in $\mathcal{E}$ entering $X^i_I.$
Thus, since $\{\mathcal{B}_r\}_{r \in R}$ is a packing and by the above lower bound on $|R_i|$, we have 
$e_{\mathcal{E}}(\{X_I^i\}_1^{\ell})+\sum_{i=1}^\ell d^{-}_{\mathcal{A}}({\sf X}^{i})\geq \sum_{i=1}^\ell|R_i|\geq \sum_{i=1}^\ell(r_{\mathcal{M}}(P^{\mathcal{H}}_C\cap R)-r_{\mathcal{M}}(X^{i}_W\cap R)),$
%
so \eqref{4cond} holds.
\medskip

For {\bf sufficiency}, let $(${\boldmath$\mathcal{H}$} $=(V \cup R, \mathcal{A}\cup \mathcal{E}),$ {\boldmath$\mathcal{M}$} $=(R,r_{\mathcal{M}}))$ be a minimum counterexample  with respect to $|V|$. Obviously, $V \neq \emptyset$.  
Let {\boldmath $C$}  be the vertex set of a strongly connected component of $\mathcal{H}-R$ that has no dyperedge
 leaving. Since each $r\in R$ is a root, $C$ exists. 

Let {\boldmath $\mathcal{H}_1$} $=(${\boldmath $V_1$}$\cup R, ${\boldmath $\mathcal{A}_1$}$\cup${\boldmath $ \mathcal{E}_1$}$)=\mathcal{H}-C.$ Note that $(\mathcal{H}_1,\mathcal{M})$ is a matroid-rooted mixed hypergraph.
%
%
The fact that  $d_\mathcal{A}^+(C)=d_\mathcal{E}(C)=0$ implies that for all $X\subseteq V_1\cup R$, we have $P^{\mathcal{H}_1}_X=P^\mathcal{H}_X$,   for every subpartition $\mathcal{P}$ of $V_1\cup R$, we have $e_{\mathcal{E}}(\mathcal{P})=e_{\mathcal{E}_1}(\mathcal{P})$, and for every biset ${\sf X}$ on $V_1\cup R$, we have $d_{\mathcal{A}_1}^{-}({\sf X})=d_{\mathcal{A}}^{-}({\sf X})$. Then, since $\mathcal{H}$ satisfies \eqref{4cond}, so does $\mathcal{H}_1.$ Hence, by the minimality of $\mathcal{H},$ 
%
$(\mathcal{H}_1,\mathcal{M})$ has a matroid-reachability-based packing of mixed hyperarborescences $\{${\boldmath $\mathcal{B}^1_r$}$\}_{r \in R}$. By definition, $\mathcal{B}^1_r$ can be made into an $r$-arborescence {\boldmath$B^1_r$} by a directed trimming for all $r \in R$ such that  $\{r \in R:v \in V(B^1_r)\}$ is a basis of $P^{\mathcal{H}_1}_v\cap R=P^{\mathcal{H}}_v \cap R$ in $\mathcal{M}$ for all $v \in V_1$. 

We now define  a matroid-rooted mixed hypergraph $(\mathcal{H}_2,\mathcal{M}_2)$ which depends on the arborescences $\{B^1_r\}_{r \in R}.$ Let {\boldmath $R_2$} $=P_C^{\mathcal{H}}\cap R$, {\boldmath $\mathcal{M}_2$} the restriction of $\mathcal{M}$ to $R_2$ and  let {\boldmath $\mathcal{H}_2$} $=(V_2\cup R_2,\mathcal{A}_2\cup \mathcal{E}_2)$ be obtained from $\mathcal{H}[C]$ by adding a set  {\boldmath $T$} of new vertices {\boldmath $t_a$} for all $a \in\rho^-_{\mathcal{A}}(C)$ and the vertex set $R_2$ and by adding dyperedges {\boldmath $a'$} $=((tail(a)\cap C)\cup t_a, head(a))$ for all $t_a \in T,$  the arcs $rt_a$ for all $r \in R_2, t_a \in T$ with $tail(a)\cap V(B^1_r)\neq\emptyset$ and $r_{\mathcal{M}_2}(R_2)$ parallel arcs $head(a)t_a$ for all $t_a \in T$.

\begin{lemma}\label{reachmixedd''pack}
$(\mathcal{H}_2,\mathcal{M}_2)$ satisfies \eqref{condmatmixhyp}.
\end{lemma}

\begin{proof}
Let {\boldmath $\{{\sf X}^i\}_1^\ell$} be a biset subpartition of $V_2=C\cup T$ with 
$X_W^i=N_{\mathcal{H}_2}^-(X_I^i)\cap R_2$ for all $i=1,\ldots,\ell$.
We may suppose that there exists {\boldmath$j$} $\in \{0,\ldots,\ell\}$ such that  $X_I^i\cap C\neq\emptyset$ for all $i\in \{1,\ldots,$ {\boldmath$j$}$\}$ and $X_I^i\cap C=\emptyset$ for all $i\in \{j+1,\ldots,\ell\}.$ 

First let {\boldmath $i$} $\in \{j+1,\ldots,\ell\}$. By $X_I^i\cap V_2\neq\emptyset$,  there exists  some $t_{a} \in X_I^i\cap T$. It follows that $a$ enters $C$, so $head(a)\in C.$ Since $X_I^i\cap C=\emptyset$ and $X_W^i\subseteq R,$ we obtain $C\subseteq V-X_O^i$. Thus the number of dyperedges in ${\mathcal{A}_2}$ entering ${\sf X}^i$ is at least the number  of arcs in ${\mathcal{A}_2}$ from $head(a)$ to $t_a$ which is $r_{\mathcal{M}_2}(R_2)$. Hence  
\begin{equation}\label{liugfdsjhlgcl}
0\ge r_{\mathcal{M}_2}( R_2)-d_{\mathcal{A}_2}^{-}({\sf X}^i)\ge r_{\mathcal{M}_2}( R_2)-r_{\mathcal{M}_2}(X_W^i)-d_{\mathcal{A}_2}^{-}({\sf X}^i).
\end{equation}

Let now {\boldmath $i$} $\in \{1,\ldots,j\}.$ Let {\boldmath$R'$} $=span_{\cal M}(X_W^i),$ {\boldmath $Y^i$} $=(V\cup R)-({U}^{\mathcal{H}}_{R-R'}\cup C)$ and let the biset  {\boldmath ${\sf Z}^{i}$} be defined as $((X_I^{i}\cap C)\cup Y^i,X_I^{i}\cap C)$.   
Note that $Z^i_I=X_I^i\cap C$ and $Z^{i}_W\cap R=Y^i\cap R=R-(R-R')=R'$, so 
\begin{equation}\label{knsijdsuy}
r_{\mathcal{M}}(Z_W^{i}\cap R)=r_{\mathcal{M}}(R')=r_{\mathcal{M}_2}(X_W^i). 
\end{equation}

Since no dyperedge and no hyperedge leaves ${U}^{\mathcal{H}}_{R-R'}\cup C=(V \cup R)-Y^i,$ we have
\begin{equation}\label{dkbjh}
Z_W^i=Y^i=P^{\mathcal{H}}_{Y^i}=P^{\mathcal{H}}_{Z_W^i}. 
\end{equation}

\begin{claim}\label{conddxdyhyp}
$d_{\mathcal{A}_2}^{-}({\sf X}^{i})\geq d_{\mathcal{A}}^{-}({\sf Z}^{i}).$
 \end{claim}
\begin{proof}
Let {\boldmath $a$} $ \in \rho_{\mathcal{A}}^{-}({\sf Z}^{i})$ and {\boldmath $v$} $=head(a)$. Then $v \in Z_I^i=X_I^i \cap C$. If $tail(a)\subseteq  C$, then $a\in \rho_{\mathcal{A}_2}^{-}({\sf X}^{i})$. Otherwise, the dyperedge $a$ enters $C$ and hence $t_a$ exists in $T.$ If $(tail(a)-Z_O^{i})\cap C \neq \emptyset$, then $a' \in \rho_{\mathcal{A}_2}^{-}({\sf X}^{i})$. In the remaining case, as $a$ enters the biset ${\sf Z}^i$, there exists some {\boldmath$u$} $\in tail(a)-Z_O^{i}- C=tail(a)-Y^{i}- C$. Then $u \in U^{\mathcal{H}}_{\bar r}$ for some {\boldmath ${\bar r}$} $\in R-R'$. If $t_a \in X_I^{i}$, then, since $\{r \in R: u \in V(B^1_r)\}$ is a basis of $P_u^{\mathcal{H}}\cap R$ in $\mathcal{M}$, we obtain 
\begin{align*}
{\bar r}\notin R'&=span_{\mathcal{M}}(X_W^i)\\
&=span_{\mathcal{M}}(N_{\mathcal{H}_2}^-(X_I^i)\cap R_2)\\
&\supseteq span_{\mathcal{M}}(N_{\mathcal{H}_2}^-(t_a)\cap R_2)\\
&= span_{\mathcal{M}}(\{r \in R_2:tail(a)\cap V(B^1_r)\neq\emptyset\})\\
&\supseteq span_{\mathcal{M}}(\{r \in R: u \in V(B^1_r)\})\\
&\supseteq P_u^{\mathcal{H}}\cap R\supseteq \{{\bar r}\},
\end{align*}
 a contradiction.
  It follows that $t_a \notin X^{i}_I$ and hence, by $X_W^i\subseteq R_2$, we have  $t_a \notin X^{i}_O$. This yields that $a' \in \rho_{\mathcal{A}_2}^{-}({\sf X}^{i})$. Note that for distinct dyperedges in ${\mathcal{A}}$ entering the biset ${\sf Z}^i$ distinct dyperedges in ${\mathcal{A}_2}$ entering the biset ${\sf X}^i$ have been found above, so the claim follows.
\end{proof}

Since $Z_W^i\cap C=\emptyset$ for all $i\in\{1,\dots,j\}$ and $\{X_I^i\}_1^{\ell}$ is a  subpartition of $C\cup T$, we have that  $\{{\sf Z}^i\}_1^j$ is a biset subpartition of $C.$ Then, by \eqref{dkbjh}, we may apply \eqref{4cond} for $\{{\sf Z}^i\}_1^j$, so, by  definition of $R_2$, \eqref{knsijdsuy}, Claim \ref{conddxdyhyp} and \eqref{liugfdsjhlgcl}, we have
$e_{\mathcal{E}_2}(\{X_I^i\}_1^{\ell})	=e_{\mathcal{E}_2}(\{X_I^i\}_1^j)=e_{\mathcal{E}}(\{Z_I^i\}_1^j)	\geq \sum_{i=1}^j(r_{\mathcal{M}}(P^{\mathcal{H}}_C\cap R)-r_{\mathcal{M}}(Z^{i}_W\cap R)-d^{-}_{\mathcal{A}}({\sf Z}^{i}))\geq \sum_{i=1}^j(r_{\mathcal{M}_2}(R_2)-r_{\mathcal{M}_2}(X_W^i)-d_{\mathcal{A}_2}^{-}({\sf X}^{i}))\geq \sum_{i=1}^{\ell}(r_{\mathcal{M}_2}(R_2)-r_{\mathcal{M}_2}(X_W^i)-d_{\mathcal{A}_2}^{-}({\sf X}^{i})),$
that is \eqref{condmatmixhyp} holds for $\{{\sf X}^i\}_1^\ell$ in $(\mathcal{H}_2,\mathcal{M}_2)$. This completes the proof of Lemma \ref{reachmixedd''pack}.
%
\end{proof}

By Lemma \ref{reachmixedd''pack}, $(\mathcal{H}_2,\mathcal{M}_2)$ has a matroid-based packing of mixed hyperarborescences  $\{${\boldmath $\mathcal{B}^2_r$}$\}_{r \in R_2}$ with $r$-arborescences $\{${\boldmath $B^2_r$}$\}_{r \in R_2}$ as directed trimmings like in the definition. We finally give a construction of a  packing of the desired form in $(\mathcal{H},\mathcal{M})$ using the packings $\{\mathcal{B}^1_r\}_{r\in R}$ and $\{\mathcal{B}^2_r\}_{r\in R_2}$, which yields a contradiction.

\begin{lemma}\label{reachmixeddpack}
$(\mathcal{H},\mathcal{M})$ has a matroid-reachability-based packing of mixed hyperarborescences.
\end{lemma}

\begin{proof}
For $r \in R-R_2$, let {\boldmath ${B}_r$} $={B}^1_r$  and for $r \in R_2$, let {\boldmath ${B}_r$} be obtained from the union of ${B}^1_r$ and ${B}^2_r-R_2-T$ by adding an arc $uv$ for all arcs $t_av$ of ${B}^2_r$ for some $u\in tail(a)\cap V(B^1_r).$ As in the proof of Theorem \ref{matreach}, we can see that $\{B_r\}_{r \in R}$ is a packing of arborescences such that the root of $B_r$ is  $r$ for all $r\in R$ and $\{r \in R:v \in V(B_r)\}$ is a basis of $P^{\mathcal{H}}_v\cap R$ in $\mathcal{M}$ for all $v \in V$. 

Finally, for $r \in R-R_2$, let {\boldmath $\mathcal{B}_r$} $=\mathcal{B}^1_r$  and for $r \in R_2$, let {\boldmath $\mathcal{B}_r$} be obtained from $\mathcal{B}^1_r$ and $\mathcal{B}^2_r-R_2-T$ by adding the dyperedge $a\in \mathcal{A}$ for all $a' \in  \mathcal{A}(\mathcal{B}^2_r)$. The above argument shows that this is a packing of mixed hyperarborescences in $\mathcal{H}$ (with arborescences $\{B_r\}_{r \in R}$ as  directed trimmings) with the desired properties.
\end{proof}

Lemma \ref{reachmixeddpack} contradicts the fact that $(\mathcal{H},\mathcal{M})$ is a counterexample and hence  the proof of Theorem \ref{matmixhypreach} is complete.
\end{proof}

\section{Algorithmic aspects}\label{alg}

This section deals with the algorithmic consequences of our proofs. 
First in Subsection \ref{algunw}, we provide algorithms for two unweighted problems in digraphs, namely finding a packing of reachability arborescences and finding a matroid-reachability-based packing of arborescences. After, we give an even stronger result in the most general setting: we show how to compute a matroid-reachability-based packing of mixed hyperarborescences minimizing a given weight function in polynomial time. In order to do so, we first provide a polynomial time algorithm for finding a  matroid-based packing of mixed hyperarborescences of minimum weight in a mixed hypergraph in Subsection \ref{mb}. Finally, we show how to use this algorithm and the proof of Theorem \ref{matmixhypreach} to find a matroid-reachability-based packing of mixed hyperarborescences of minimum weight in a mixed hypergraph in polynomial time in Subsection \ref{mainalgo}.

\subsection{Unweighted algorithms}\label{algunw}
For the basic case, we show that our proof of Theorem \ref{3japt}  yields a  polynomial time algorithm for finding the desired packing of reachability arborescences. We acknowledge that so does the original proof in \cite{kkt}. We first mention that the packings in Theorem \ref{edstrong} can be found in polynomial time, either following  the proof of Frank (Theorem 10.2.1 in \cite{book}) or using the efficient algorithm of Gabow \cite{gabow}. Using this algorithm and the notation of Subsection \ref{lfbdjhfvhd}, we now turn our proof of Theorem \ref{3japt} into a recursive polynomial time algorithm for finding the desired packing of arborescences. It is clear that the vertex set $C$ of a strongly connected component of $D$ that has no arc leaving can be found in polynomial time. Recursively, we first find the arborescences $B^1_r$ $(r\in R)$ in the smaller instance $D_1$ in polynomial time. As the size of $D_2$ is polynomial in the size of $D$, we can apply the algorithm mentioned above to obtain the arborescences $B^2_r$ $(r\in R_2)$ in polynomial time. The obtained arborescences can be merged efficiently to obtain the arborescences $B_r$ $(r\in R)$ as in 
the proof of Theorem \ref{3japt}. Clearly, this algorithm  runs in polynomial time and provides the  packing $\{B_r\}_{r \in R}$ of reachability $r$-arborescences in $D.$

For the matroidal case, we suppose that a polynomial time  independence oracle for $\mathcal{M}$ is given and show that our proof of Theorem \ref{matreach} yields a polynomial time algorithm for finding the matroid-reachability-based packing of arborescences. We acknowledge that so does the original proof in \cite{k}. We first mention that the packings in Theorem \ref{mat} can be found in polynomial time  if a polynomial time  independence oracle for $\mathcal{M}$ is given as shown in \cite{dns}. It is easy to see that the proof of Theorem \ref{matbranch} yields a polynomial time algorithm if a polynomial time  independence oracle for $\mathcal{M}$ is given.  By similar arguments as before and the fact that we easily obtain polynomial time independence oracles for all matroids considered, we obtain that the proof of Theorem \ref{matreach} can be turned into a recursive polynomial time algorithm.

\subsection{Weighted algorithm for matroid-based packings of mixed hyperarborescences}\label{mb}

This section is dedicated to giving a polynomial time algorithm for finding a  matroid-based packing of mixed hyperarborescences minimizing a given weight function in a mixed hypergraph. We do so by modeling the hyper- and dyperedge sets of such packings as the intersection of two matroids. This extends a method from \cite{kst}.

We need the following result of Fortier et al. \cite{fklst} that provides a polynomial time algorithm to find a matroid-based packing of  hyperarborescences in a matroid-rooted dypergraph.

\begin{theorem}[\cite{fklst}]\label{algofkmst}
Let $(\mathcal{D},\mathcal{M})$ be a matroid-rooted dypergraph with a polynomial time independence oracle for $\mathcal{M}$ being available. Then, we can decide in polynomial time whether a matroid-based packing of  hyperarborescences exists in $(\mathcal{D}, \mathcal{M})$. Further, if such a packing exists, then we can compute one in polynomial time.
\end{theorem}

\subsubsection{Relevant matroids}

We here give the necessary matroids for this section.
We first define Katoh-Tanigawa matroids, then extended Katoh-Tanigawa matroids and finally entering matroids.

The first matroid is a straightforward extension of the construction of the graphic case in \cite{kst}. The authors of \cite{kst} also rely on ideas in \cite{kt}. It further generalizes the notion of hypergraphic matroids which was introduced by Lorea \cite{l} and later studied by Frank, Kir\'aly and Kriesell \cite{fkk2}. Let $(\mathcal{H}=(V \cup R, \mathcal{A}\cup\mathcal{E}),\mathcal{M}=(R,r_{\mathcal{M}}))$ be a simply matroid-rooted  mixed hypergraph with a polynomial time independence oracle for $\mathcal{M}$ being available. 
We consider the function $b_{(\mathcal{H},\mathcal{M})}:2^{\mathcal{A}\cup\mathcal{E}}\rightarrow \mathbb{Z}$ defined as 
\smallskip

\centerline{{\boldmath$b_{(\mathcal{H},\mathcal{M})}(\mathcal{X})$} $=r_{\mathcal{M}}(R)(|V_\mathcal{X}|-1)+r_{\mathcal{M}}(R_\mathcal{X})$ for all $\mathcal{X}\subseteq\mathcal{A}\cup\mathcal{E}$.} 

\begin{lemma}\label{submod}
The function $b_{(\mathcal{H},\mathcal{M})}$ is submodular on $2^{\mathcal{A}\cup\mathcal{E}}.$
\end{lemma}

\begin{proof}
Since the cardinality function  is modular, $V_{\mathcal{Y}\cap\mathcal{Z}}\subseteq V_{\mathcal{Y}}\cap V_{\mathcal{Z}}$ and $V_{\mathcal{Y}\cup\mathcal{Z}}=V_{\mathcal{Y}}\cup V_{\mathcal{Z}}$ for all $\mathcal{Y},\mathcal{Z}\subseteq\mathcal{A}\cup\mathcal{E}$, we have that  the function $r_{\mathcal{M}}(R)(|V_{\mathcal{X}}|-1)$ is submodular.
Since the function $r_{\mathcal{M}}(R_\mathcal{X})$ is submodular and non-decreasing, 
$R_{\mathcal{Y}\cap\mathcal{Z}}\subseteq R_{\mathcal{Y}}\cap R_{\mathcal{Z}}$ and $R_{\mathcal{Y}\cup\mathcal{Z}}=R_{\mathcal{Y}}\cup R_{\mathcal{Z}}$ for all $\mathcal{Y},\mathcal{Z}\subseteq\mathcal{A}\cup\mathcal{E}$, we have that the function $r_{\mathcal{M}}(R_{\mathcal{X}})$ is submodular. It follows that the function $b_{(\mathcal{H},\mathcal{M})}$ is  submodular.
\end{proof}

To define a matroid by the function $b_{(\mathcal{H},\mathcal{M})}$, we need the following result of Edmonds \cite{e1} (see also in \cite[Section 13.4]{book}).

\begin{theorem}[\cite{e1}]\label{matroid}
If $b$ is an integer valued, non-decreasing and submodular function on $2^S$ for a ground set $S$  with $b(X)\geq 0$ for all nonempty $X \subseteq S$,  then $\mathcal{I}_{b}=\{X\subseteq S:b(Y)\ge |Y| \text{ for all nonempty }  Y \subseteq X\}$ is the set of independent sets of a matroid on $S$.
\end{theorem}


We now define a set family on $\mathcal{A}\cup\mathcal{E}$ by the function $b_{(\mathcal{H},\mathcal{M})}$ as follows: 
\smallskip

\centerline{{\boldmath$\mathcal{I}_{(\mathcal{H},\mathcal{M})}$} $=\{\mathcal{X} \subseteq \mathcal{A}\cup \mathcal{E}:b_{(\mathcal{H},\mathcal{M})}(\mathcal{Y})\geq |\mathcal{Y}|\text{ for all nonempty }  \mathcal{Y} \subseteq \mathcal{X}\}$.}
\smallskip

By Lemma \ref{submod} and Theorem \ref{matroid}, we have the following corollary. It was proven for the graphic case in a slightly different form in \cite{kst}.
\begin{theorem}\label{inde}
$\mathcal{I}_{(\mathcal{H},\mathcal{M})}$ is the set of independent sets of a matroid on $\mathcal{A}\cup \mathcal{E}$.
\end{theorem}
%

We refer to the matroid defined in Theorem \ref{inde} as the {\it Katoh-Tanigawa matroid} {\boldmath${\mathcal K}_{(\mathcal{H},\mathcal{M})}$}. For the algorithmic application, we need to show that a polynomial time  independence oracle for ${\mathcal K}_{(\mathcal{H},\mathcal{M})}$ is available. As mentioned in \cite{kst}, a method to find a polynomial time independence oracle for ${\mathcal K}_{(\mathcal{H},\mathcal{M})}$ in the graphic case has been given in \cite{kt}. For the sake of completeness, we here show how to obtain a polynomial time independence oracle for ${\mathcal K}_{(\mathcal{H},\mathcal{M})}$. To do so we use  submodular function minimization  (\cite{separation, Iwata, s}). In fact we must minimize a submodular function on nonempty sets, so we need the following result of Goemans and Ramakrishnan   \cite{goemans}.

\begin{theorem}[\cite{goemans}]\label{minintsubm}
Let $S$ be a ground set and $b:2^S \rightarrow \mathbb{Z}$ a submodular function such that $b(X)$ can be computed in polynomial time for each $X\subseteq S.$ Then $\min_{\emptyset\neq X \subseteq S}b(X)$ can be found in polynomial time.
\end{theorem}

We are now ready to give the desired polynomial time independence oracle for the Katoh-Tanigawa matroid.

\begin{lemma}\label{orachyp}
We can decide in polynomial time whether $\mathcal{X}\subseteq \mathcal{A}\cup\mathcal{E}$ is independent in the Katoh-Tanigawa matroid ${\mathcal K}_{(\mathcal{H},\mathcal{M})}$.
\end{lemma}

\begin{proof}
Let {\boldmath$b'({\mathcal{X}})$} $=b_{(\mathcal{H},\mathcal{M})}({\mathcal{X}})-|{\mathcal{X}}|$ for all $\mathcal{X}\subseteq \mathcal{A}\cup\mathcal{E}.$ As a polynomial time independence oracle for $\mathcal{M}$ is available, $b_{(\mathcal{H},\mathcal{M})}({\mathcal{X}})$ and hence $b'(\mathcal{X})$ can be computed in polynomial time for all $\mathcal{X}\subseteq \mathcal{A}\cup\mathcal{E}.$ Since, by Lemma \ref{submod}, the function $b_{(\mathcal{H},\mathcal{M})}$ is submodular, so is $b'.$ Therefore, by Theorem \ref{minintsubm}, we can compute $\min_{\emptyset\neq\mathcal{Y}\subseteq\mathcal{X}}b'(\mathcal{Y})$ in polynomial time. Further, by definition, $\mathcal{X}$ is independent in $\mathcal{K}_{(\mathcal{H},\mathcal{M})}$ if and only if $\min_{\emptyset\neq\mathcal{Y} \subseteq \mathcal{X}}b'(\mathcal{Y})\geq 0$. This finishes the proof.
\end{proof}
 
We need another matroid which is obtained by modifying the Katoh-Tanigawa matroid via parallel extensions. The ground set of the {\it extended Katoh-Tanigawa matroid} {\boldmath${\mathcal K}^{ex}_{(\mathcal{H},\mathcal{M})}$} is $\mathcal{A}\cup \mathcal{A}_\mathcal{E}$ and it is constructed from the Katoh-Tanigawa matroid $\mathcal{K}_{(\mathcal{H},\mathcal{M})}$ by replacing every $e \in \mathcal{E}$ by $|e|$ parallel copies of itself, associating them to the elements in $\mathcal{A}_e$. The following is an immediate corollary of the construction of the extended Katoh-Tanigawa matroid and Lemma \ref{orachyp}.

\begin{corollary}\label{oracmixhyp}
We can decide in polynomial time whether $\mathcal{X}\subseteq \mathcal{A}\cup \mathcal{A}_\mathcal{E}$ is independent in the extended Katoh-Tanigawa matroid ${\mathcal K}^{ex}_{(\mathcal{H},\mathcal{M})}$.
\end{corollary}

We finally adapt a matroid construction already used  by Edmonds to rooted dypergraphs. Given a simply rooted dypergraph $\mathcal{D}=(V \cup R, \mathcal{A})$ and $k\in \mathbb{Z}_+$, the {\it entering matroid} {\boldmath$\mathcal{M}^{\oplus}_{(\mathcal{D},k)}$}  is the direct sum of the uniform matroids of rank $\min\{k,d_\mathcal{A}^-(v)\}$ on $\rho_{\mathcal{A}}^-(v)$ for all $v \in V$. Observe that the ground set of this matroid is $\mathcal{A}$ as $R$ is a set of roots.

\subsubsection{Characterizations via matroid intersection}

We first review a slight extension of a result of \cite{kst} characterizing matroid-based packings of hyperarborescences in simply rooted dypergraphs as the intersection of two matroids. We then conclude a generalization to matroid-based packings of mixed hyperarborescences from it.

The following result was proven in \cite[Theorem 4.1]{kst} for the digraphic case in a slightly different form. Its proof can be literally generalized to the case of dypergraphs.

\begin{theorem}\label{dyper}
Let $(\mathcal{D}=(V\cup R, {\cal A}),\mathcal{M}=(R,r_{\mathcal{M}}))$ be a simply matroid-rooted dypergraph and $k=r_{\mathcal{M}}(R)$. Then the dyperedge sets of the matroid-based packings of hyperarborescences in $(\mathcal{D},\mathcal{M})$ are exactly the common independent sets of size $k|V|$ of the Katoh-Tanigawa matroid $\mathcal{K}_{(\mathcal{D},\mathcal{M})}$ and the entering matroid $\mathcal{M}^{\oplus}_{(\mathcal{D},k)}$.
\end{theorem}
 
 Recall that a packing of hyperarborescences in the directed extension $\mathcal{D}_{\mathcal{H}}$ of a rooted mixed hypergraph $\mathcal{H}$ is {\it feasible} if it contains at most one dyperedge of the bundle $\mathcal{A}_e$ for  every hyperedge $e$ of $\mathcal{H}$. The following is an easy corollary of Theorem \ref{dyper}.
 
\begin{corollary}\label{vtgiil}
 Let $(\mathcal{H}=(V \cup R, \mathcal{A}\cup \mathcal{E}),\mathcal{M}=(R,r_{\mathcal{M}}))$ be a simply matroid-rooted mixed hypergraph,  $k$ $=r_{\mathcal{M}}(R)$ and $\mathcal{D}_{\mathcal{H}}=(V \cup R, \mathcal{A}\cup \mathcal{A}_\mathcal{E})$ the directed extension of $\mathcal{H}$.The dyperedge sets of the feasible matroid-based packings of hyperarborescences in $(\mathcal{D}_{\mathcal{H}},\mathcal{M})$ are exactly the common independent sets of size $k|V|$ of the extended  Katoh-Tanigawa matroid $\mathcal{K}^{ex}_{({\mathcal{H}},\mathcal{M})}$ and the entering matroid $\mathcal{M}^{\oplus}_{(\mathcal{D}_{\mathcal{H}},k)}$.
\end{corollary}

\begin{proof}
First let {\boldmath$\mathcal{A}'$} be the dyperedge set of a feasible matroid-based packing of hyperarborescences in $(\mathcal{D}_{\mathcal{H}},\mathcal{M})$. By Theorem \ref{dyper}, $\mathcal{A}'$ is a common independent set of size $k|V|$ of the Katoh-Tanigawa matroid ${\mathcal K}_{(\mathcal{D}_{\mathcal{H}},\mathcal{M})}$ and  the entering matroid $\mathcal{M}^{\oplus}_{(\mathcal{D}_{\mathcal{H}},k)}$. As the packing is feasible, $\mathcal{A}'$ contains at most one dyperedge of the bundle $\mathcal{A}_e$ for all $e \in \mathcal{E}$. It follows that $\mathcal{A}'$ is also independent in the extended  Katoh-Tanigawa matroid $\mathcal{K}^{ex}_{({\mathcal{H}},\mathcal{M})}$. 

Now let {\boldmath$\mathcal{A}'$} be a common independent set of size $k|V|$ of the extended  Katoh-Tanigawa matroid $\mathcal{K}^{ex}_{({\mathcal{H}},\mathcal{M})}$ and the entering matroid $\mathcal{M}^{\oplus}_{(\mathcal{D}_{\mathcal{H}},k)}$. It follows that $\mathcal{A}'$ is also an independent set  of the Katoh-Tanigawa matroid  ${\mathcal K}_{(\mathcal{D}_{\mathcal{H}},\mathcal{M})}$. Then, by Theorem \ref{dyper}, $\mathcal{A}'$ is the dyperedge set of a matroid-based packing of hyperarborescences in $(\mathcal{D}_{\mathcal{H}},\mathcal{M})$. As $\mathcal{A}'$ is independent in the extended  Katoh-Tanigawa matroid $\mathcal{K}^{ex}_{({\mathcal{H}},\mathcal{M})}$, $\mathcal{A}'$  contains at most one dyperedge of the bundle $\mathcal{A}_e$ for all $e \in \mathcal{E}$ and so  the packing is feasible.
\end{proof}

 It is easy to see  the following relation between matroid-based packings of mixed hyperarborescences in $(\mathcal{H},\mathcal{M})$ and feasible matroid-based packings of hyperarborescences in $(\mathcal{D}_{\mathcal{H}},\mathcal{M})$.

 \begin{lemma}\label{vdgdfze}
Let $(\mathcal{H},\mathcal{M})=(V \cup R, \mathcal{A} \cup \mathcal{E})$ be a matroid-rooted mixed hypergraph. Then for any $\mathcal{A}' \subseteq \mathcal{A}$ and $\mathcal{E}' \subseteq \mathcal{E}$ the following are equivalent:
\begin{itemize}
	\item[(a)] There exists a matroid-based packing $\mathcal{B}$ of mixed hyperarborescences in $(\mathcal{H},\mathcal{M})$ such that the set of dyperedges of $\mathcal{B}$ is $\mathcal{A}'$ and the set of hyperedges of $\mathcal{B}$ is $\mathcal{E}'$,
	\item[(b)] There exists an orientation $\vec{\mathcal{E}}$ of $\mathcal{E}$ such that in the matroid-rooted dypergraph obtained there exists a matroid-based packing of hyperarborescences whose dyperedge set is $\mathcal{A}'\cup\vec{\mathcal{E}'}.$
	\item[(c)]  There exists a feasible matroid-based  packing $\mathcal{B}$ of hyperarborescences in $(\mathcal{D}_\mathcal{H},\mathcal{M})$ such that the set of dyperedges of $\mathcal{A}$ contained in $\mathcal{B}$ is $\mathcal{A}'$ and $\{e\in\mathcal{E}: \mathcal{B} \text{ contains a dyperedge of }\mathcal{A}_e\}=\mathcal{E}'.$
\end{itemize}
\end{lemma}

Observe that Lemma \ref{vdgdfze} and Corollary \ref{vtgiil} establish a connection between matroid intersection and matroid-based packings of mixed hyperarborescences.

\subsubsection{The algorithm}

To provide the desired algorithm, we  need the following result on matroid intersection of Edmonds \cite{e2}.

\begin{theorem}[\cite{e2}]\label{matroidintersection}
Let ${\mathcal M}_1$ and ${\mathcal M}_2$ be two matroids on a common ground set $S$  with polynomial time independence oracles for ${\mathcal M}_1$ and ${\mathcal M}_2$ being available, $w:S\rightarrow \mathbb{R}$ a weight function and $\mu\in \mathbb{Z}_+$. It can be decided in polynomial time if a common independent set of size $\mu$ of ${\mathcal M}_1$ and ${\mathcal M}_2$ exists. Further, if such a set exists, then one of minimum weight with respect to $w$ can be computed in polynomial time.
\end{theorem}

We are now ready to show how Theorems \ref{vtgiil} and \ref{matroidintersection} imply the 
following  theorem which is the 
main result of this section.

\begin{theorem}\label{algo2}
Let $(\mathcal{H}=(V \cup R, \mathcal{A}\cup \mathcal{E}),\mathcal{M}=(R,r_{\mathcal{M}}))$ be a matroid-rooted mixed hypergraph with a polynomial time independence oracle for $\mathcal{M}$ being available and $w: \mathcal{A}\cup \mathcal{E} \rightarrow \mathbb{R}$ a weight function. We can decide in polynomial time whether a matroid-based packing of mixed hyperarborescences in $(\mathcal{H}, \mathcal{M})$ exists. Further, if such a packing exists, then we can compute one of minimum weight with respect to $w$ in polynomial time.
\end{theorem}

\begin{proof} Note that we may suppose that $(\mathcal{H},\mathcal {M})$ is simply rooted. Indeed, using the polynomial time algorithm for the simply rooted case, the proof of Theorem \ref{matmixhypbranch} can be turned into a polynomial time algorithm   for the  rooted case by observing that the polynomial time  independence oracle for $\mathcal{M}$ trivially provides a polynomial time  independence oracle for $\mathcal{M}'$. 

Let $\mathcal{D}_{\mathcal{H}}=(V\cup R,\mathcal{A}\cup \mathcal{A}_{\mathcal{E}})$ be the directed extension of $\mathcal{H}$, {\boldmath$k$} $=r_{\mathcal{M}}(R)$ and define {\boldmath$w'$}$:\mathcal{A}\cup \mathcal{A}_{\mathcal{E}}\rightarrow \mathbb{R}$ by $w'(a)=w(a)$ for all $a \in \mathcal{A}$ and $w'(a)=w(e)$ for all $a \in \mathcal{A}_{e}$ for all $e \in\mathcal{E}$. 

By Lemma \ref{vdgdfze}, it suffices to decide in polynomial time whether a feasible matroid-based  packing of hyperarborescences  exists in $(\mathcal{D}_\mathcal{H},\mathcal{M})$ and to find the dyperedge set of one of minimum weight with respect to $w'$ in polynomial time if such a packing exists. Indeed, by Lemma \ref{vdgdfze}, the dyperedge set of this feasible matroid-based  packing of hyperarborescences in $(\mathcal{D}_\mathcal{H},\mathcal{M})$ is the dyperedge set of a matroid-based packing of hyperarborescences in $(\vec{\mathcal{H}},\mathcal {M})$ where $\vec{\mathcal{H}}$ is an orientation of $\mathcal{H}.$
Then, by Theorem \ref{algofkmst}, from the dyperedge set of the packing we can find the packing itself in polynomial time. Finally, by  Lemma \ref{vdgdfze}, omitting the orientations of the dyperedges in $\mathcal{A}_{\mathcal{E}}$ of the packing, we obtain a matroid-based packing of mixed hyperarborescences in $(\mathcal{H},\mathcal {M}).$

By Corollary \ref{vtgiil}, it is enough to decide in polynomial time whether a common independent set of size $k|V|$ of the extended  Katoh-Tanigawa matroid $\mathcal{K}^{ex}_{({\mathcal{H}},\mathcal{M})}$ and the entering matroid $\mathcal{M}^{\oplus}_{(\mathcal{D}_{\mathcal{H}},k)}$ exists, and if such a set exists to find one of  minimum weight with respect to $w'$ in polynomial time. 
Since a polynomial time independence oracle for $\mathcal{K}^{ex}_{({\mathcal{H}},\mathcal{M})}$ is available by Corollary \ref{oracmixhyp} and a polynomial time independence oracle for $\mathcal{M}^{\oplus}_{(\mathcal{D}_{\mathcal{H}},k)}$ is easily available, the theorem follows from Theorem \ref{matroidintersection}.
\end{proof}

\subsection{Weighted algorithms for matroid-reachability-based packings of mixed hyperarborescences}\label{mainalgo}

This final section is dedicated to showing how to turn the proof of the Theorem \ref{matmixhypreach} into a recursive polynomial time algorithm for finding matroid-reachability-based packings of mixed hyperarborescences of minimum weight using the algorithm obtained in Theorem \ref{algo2}. 
The following theorem is our main algorithmic result.

\begin{theorem}\label{algo3}
Given a  matroid-rooted mixed hypergraph $(\mathcal{H}=(V \cup R, \mathcal{A}\cup \mathcal{E}),\mathcal{M}=(R,r_{\mathcal{M}}))$  with a polynomial time independence oracle for $\mathcal{M}$ being available and a weight function $w: \mathcal{A}\cup \mathcal{E} \rightarrow \mathbb{R}$, we can decide in polynomial time whether a matroid-reachability-based packing of mixed hyperarborescences in $(\mathcal{H}, \mathcal{M})$ exists. Further, if such a packing exists, then we can compute one of minimum weight with respect to $w$ in polynomial time.
\end{theorem}

\begin{proof}
We use the notation of the proof of the Theorem \ref{matmixhypreach}. 
It is known that we can  find in polynomial time the vertex sets of the strongly connected components of $\mathcal{H}$ and hence one strongly connected component $C$ that has no dyperedge 
leaving can be found in polynomial time.

We define the weight functions {\boldmath$w_1$}$:\mathcal{A}_1\cup \mathcal{E}_1\rightarrow \mathbb{R}$ and {\boldmath$w_2$}$:\mathcal{A}_2\cup \mathcal{E}_2\rightarrow \mathbb{R}$ as follows. Let $w_1(e)=w(e)$ for all $e \in \mathcal{A}_1\cup\mathcal{E}_1$ and let  $w_2(e)=w(e)$ for all $e \in \mathcal{A}(C)\cup\mathcal{E}(C)$, $w_2(a')=w(a)$ for all $a' \in \rho_{\mathcal{A}_2}^-(C)$ and $w_2(a)=0$ for all $a \in \rho_{\mathcal{A}_2}^-(T)$. 

Recursively, we may suppose that we can decide in polynomial time if a matroid-reachability-based packing of mixed hyperarborescences exists in $(\mathcal{H}_1,\mathcal{M})$. If such a packing does not exist, then no matroid-reachability-based packing of mixed hyperarborescences exists in $(\mathcal{H},\mathcal{M})$. If such a packing exists, then, recursively, we may suppose that one,  {\boldmath $\mathcal{B}^1$},  minimizing $w_1(\mathcal{B}^1)$ can be computed in polynomial time. By Theorem \ref{algo2}, we can decide in polynomial time if a matroid-based packing of mixed hyperarborescences exists in $(\mathcal{H}_2,\mathcal{M}_2)$. If such a packing does not exist, then no matroid-reachability-based packing of mixed hyperarborescences exists in $(\mathcal{H},\mathcal{M})$. If such a packing exists, then, one,  {\boldmath $\mathcal{B}^2$},  minimizing $w_2(\mathcal{B}^2)$ can be computed in polynomial time by Theorem \ref{algo2}. We can construct in polynomial time the matroid-reachability-based packing {\boldmath$\mathcal{B}$} of mixed hyperarborescences in $(\mathcal{H},\mathcal{M})$ by merging $\mathcal{B}^1$ and $\mathcal{B}^2$ as in the proof of Lemma \ref{reachmixeddpack}. Note that $w_1(\mathcal{B}^1)+w_2(\mathcal{B}^2)=w(\mathcal{B})$. 
We now prove the optimality of $\mathcal{B}$ in two steps.

\begin{lemma}\label{qdcfd}
There is a matroid-reachability-based packing $\mathcal{B}^*$ of mixed hyperarborescences in $(\mathcal{H},\mathcal{M})$ that minimizes $w(\mathcal{B}^*)$ and extends $\mathcal{B}^1$.
\end{lemma}
\begin{proof}
Let {\boldmath $\hat{\mathcal{B}}$} $=\{${\boldmath $\hat{\mathcal{B}}_r$}$\}_{r \in R}$ be a matroid-reachability-based packing of mixed hyperarborescences in $(\mathcal{H},\mathcal{M})$ minimizing $w(\hat{\mathcal{B}})$, {\boldmath $\hat{\mathcal{A}}_C$} $=\mathcal{A}(\hat{\mathcal{B}})\cap(\mathcal{A}(C)\cup \delta_{\mathcal{A}}^-(C))$, {\boldmath $\hat{\mathcal{E}}_C$} $=\mathcal{E}(\hat{\mathcal{B}})\cap\mathcal{E}(C)$ and {\boldmath $\hat{\mathcal{F}}_C$} $=\hat{\mathcal{A}}_C \cup \hat{\mathcal{E}}_C.$

Now consider {\boldmath $\mathcal{H}'$}$=(V \cup R, \mathcal{A}' \cup \mathcal{E'})$ with {\boldmath$\mathcal{A}'$} $=\mathcal{A}_1\cup \hat{\mathcal{A}}_C$ and {\boldmath$\mathcal{E}'$} $=\mathcal{E}_1\cup \hat{\mathcal{E}}_C$. Observe that $\hat{\mathcal{B}}$ is a matroid-reachability-based packing of mixed hyperarborescences in $(\mathcal{H}',\mathcal{M})$ and $\mathcal{B}^1$ is a matroid-reachability-based packing of mixed hyperarborescences in $(\mathcal{H}'-C,\mathcal{M})$. The proof of Theorem \ref{matmixhypreach} implies that $\mathcal{B}^1$ can be extended to a matroid-reachability-based packing of mixed hyperarborescences $\mathcal{B}^*$  in $(\mathcal{H}',\mathcal{M})$. As $\mathcal{B}^*$ is also a matroid-reachability-based packing of mixed hyperarborescences in $(\mathcal{H},\mathcal{M})$ and $\mathcal{B}^*$ extends $\mathcal{B}_1$, it suffices to prove that $w(\mathcal{B}^*)\leq w(\hat{\mathcal{B}})$ holds. 

Let {\boldmath $\hat{\mathcal{B}}_1$} be obtained from $\hat{\mathcal{B}}$  by deleting the vertices in $C$ from each $\hat{\mathcal{B}}_r$.
Then, since $d_\mathcal{A}^+(C)=d_\mathcal{E}(C)=0$, we obtain that $\hat{\mathcal{B}}_1$ is a matroid-reachability-based packing of mixed hyperarborescences in $(\mathcal{H}_1,\mathcal{M})$. This yields $w(\mathcal{B}^1)\leq w(\hat{\mathcal{B}}_1)$.

Let {\boldmath$\mathcal{F}$} $=(\mathcal{A}(\mathcal{B}^*)\cup \mathcal{E}(\mathcal{B}^*))-(\mathcal{A}(\mathcal{B}^1)\cup \mathcal{E}(\mathcal{B}^1))$. 
Since both $\mathcal{B}^*$ and  $\hat{\mathcal{B}}$ are matroid-reachability based packings of mixed hyperarborescences in $(\mathcal{H},\mathcal{M})$, we have $|\mathcal{F}|=r_\mathcal{M}(P_\mathcal{H}^C)|C|=|\hat{\mathcal{F}}_C|$. As $\mathcal{F}\subseteq \hat{\mathcal{F}}_C$, we obtain $\mathcal{F}= \hat{\mathcal{F}}_C$. 

These arguments yield $w(\mathcal{B}^*)=w(\mathcal{B}^1)+w(\mathcal{F})\leq w(\hat{\mathcal{B}}_1)+w(\hat{\mathcal{F}}_C)=w(\hat{\mathcal{B}})$ that finishes the proof.
\end{proof}

For $r\in R_2,$ let {\boldmath$\mathcal{B}^*_{2,r}$} be obtained from $\mathcal{B}^*_r$ by deleting the vertices not in $C$, by adding $r$ and the vertex set $\{t_a: a\in \rho_{\mathcal{A}}^-(C), head(a)\in U_r^{\mathcal{B}^*_r}\}$, by adding $a'$ and $rt_a$ if $a\in \rho_{\mathcal{A}}^-(C)\cap\mathcal{A}(\mathcal{B}^*_r)$ and by adding  one of the copies of the arc $head(a)t_a$ if $a\in \rho_{\mathcal{A}}^-(C)-\mathcal{A}(\mathcal{B}^*_r)$ and $head(a)\in  U_r^{\mathcal{B}^*_r}.$  As there are $r_{\mathcal{M}}(R_2)$ arcs from $head(a)$ to $t_a$, we can choose the $\mathcal{B}^*_{2,r}$'s to be hyperedge- and dyperedge-disjoint. Observe that for all $a\in \rho_{\mathcal{A}}^-(C)$, $t_a \in U_r^{\mathcal{B}^*_{2,r}}$ if and only if $head(a) \in  U_r^{\mathcal{B}^*_{2,r}}.$ Then {\boldmath $\mathcal{B}^*_2$}  $=\{\mathcal{B}^*_{2,r}\}_{r \in R_2}$ is   a matroid-based packing of  mixed hyperarborescences in $(\mathcal{H}_2,\mathcal{M}_2),$ so $w_2(\mathcal{B}_2)\leq w_2(\mathcal{B}^*_2)$. Note that $w(\mathcal{B}^*)=w_1(\mathcal{B}^1)+w_2(\mathcal{B}_2^*).$ Then, by the minimality of $w(\mathcal{B}^*)$,  we have 
 $w(\mathcal{B})\ge w(\mathcal{B}^*)=w_1(\mathcal{B}^1)+w_2(\mathcal{B}_2^*)\ge w_1(\mathcal{B}^1)+w_2(\mathcal{B}^2)=w(\mathcal{B}).$ It follows that equality holds everywhere and hence $\mathcal{B}$ is of minimum weight with respect to $w.$

Since a polynomial time independence oracle for $\mathcal{M}$ is available, polynomial time independence oracles for all matroids considered are easily available. It follows that this algorithm  runs in polynomial time and we are done.
 \end{proof}

\section*{Acknowledgments}

We wish to thank an anonymous referee for pointing out a mistake in an earlier version of this article.


\end{document}